\documentclass[11pt,english]{article}
\usepackage[T1]{fontenc}
\usepackage[latin9]{inputenc}
\usepackage{verbatim}
\usepackage{amsmath}
\usepackage{amsthm}
\usepackage{amssymb}

\makeatletter
\theoremstyle{plain}
\newtheorem{thm}{\protect\theoremname}
\theoremstyle{definition}
\newtheorem{defn}[thm]{\protect\definitionname}
\theoremstyle{plain}
\newtheorem{fact}[thm]{\protect\factname}
\theoremstyle{remark}
\newtheorem{rem}[thm]{\protect\remarkname}
\theoremstyle{plain}
\newtheorem{cor}[thm]{\protect\corollaryname}
\theoremstyle{definition}
\newtheorem{example}[thm]{\protect\examplename}
\theoremstyle{plain}
\newtheorem{prop}[thm]{\protect\propositionname}
\theoremstyle{plain}
\newtheorem{lem}[thm]{\protect\lemmaname}
\theoremstyle{remark}
\newtheorem{claim}[thm]{\protect\claimname}

\usepackage{fullpage}
\usepackage{hyperref}
\hypersetup{
	colorlinks   = true, %Colours links instead of ugly boxes
	urlcolor     = red, %Colour for external hyperlinks
	linkcolor    = blue, %Colour of internal links
	citecolor   = blue %Colour of citations, could be ``red''
}
\date{}
\makeatother

\usepackage{babel}
\providecommand{\claimname}{Claim}
\providecommand{\corollaryname}{Corollary}
\providecommand{\definitionname}{Definition}
\providecommand{\examplename}{Example}
\providecommand{\factname}{Fact}
\providecommand{\lemmaname}{Lemma}
\providecommand{\propositionname}{Proposition}
\providecommand{\remarkname}{Remark}
\providecommand{\theoremname}{Theorem}

\begin{document}
	\global\long\def\goinf{\rightarrow\infty}
	\global\long\def\gozero{\rightarrow0}
	\global\long\def\bra{\langle}
	\global\long\def\ket{\rangle}
	\global\long\def\union{\cup}
	\global\long\def\intersect{\cap}
	\global\long\def\abs#1{\left|#1\right|}
	\global\long\def\norm#1{\left\Vert #1\right\Vert }
	\global\long\def\floor#1{\left\lfloor #1\right\rfloor }
	\global\long\def\ceil#1{\left\lceil #1\right\rceil }
	\global\long\def\expect{\mathbb{E}}
	\global\long\def\e{\mathbb{E}}
	\global\long\def\r{\mathbb{R}}
	\global\long\def\n{\mathbb{N}}
	\global\long\def\q{\mathbb{Q}}
	\global\long\def\c{\mathbb{C}}
	\global\long\def\z{\mathbb{Z}}
	\global\long\def\grad{\nabla}
	\global\long\def\t{^{\prime}}
	\global\long\def\all{\forall}
	\global\long\def\eps{\varepsilon}
	\global\long\def\quadvar#1{V_{2}^{\pi}\left(#1\right)}
	\global\long\def\cal#1{\mathcal{#1}}
	\global\long\def\cross{\times}
	\global\long\def\del{\nabla}
	\global\long\def\parx#1{\frac{\partial#1}{\partial x}}
	\global\long\def\pary#1{\frac{\partial#1}{\partial y}}
	\global\long\def\parz#1{\frac{\partial#1}{\partial z}}
	\global\long\def\part#1{\frac{\partial#1}{\partial t}}
	\global\long\def\partheta#1{\frac{\partial#1}{\partial\theta}}
	\global\long\def\parr#1{\frac{\partial#1}{\partial r}}
	\global\long\def\curl{\nabla\times}
	\global\long\def\rotor{\nabla\times}
	\global\long\def\one{\mathbf{1}}
	\global\long\def\Hom{\text{Hom}}
	\global\long\def\pr#1{\text{Pr}\left[#1\right]}
	\global\long\def\almost{\mathbf{\approx}}
	\global\long\def\tr{\text{Tr}}
	\global\long\def\var{\text{Var}}
	\global\long\def\onenorm#1{\left\Vert #1\right\Vert _{1}}
	\global\long\def\twonorm#1{\left\Vert #1\right\Vert _{2}}
	\global\long\def\Inj{\mathfrak{Inj}}
	\global\long\def\inj{\mathsf{inj}}
	
	\global\long\def\g{\mathfrak{\cal G}}
	\global\long\def\f{\mathfrak{\cal F}}
	\newcommand{\bR}{\mathbb{R}}

\title{Decomposition of mean-field Gibbs distributions into product measures}

\author{Ronen Eldan\thanks{Weizmann Institute of Science. Email: ronen.eldan@weizmann.ac.il.
		Partially supported by the Israel Science Foundation, grant 715/16.} ~and Renan Gross\thanks{Weizmann Institute of Science. Email: renan.gross@weizmann.ac.il}}
\maketitle
\begin{abstract}
We show that under a low complexity condition on the gradient of a
Hamiltonian, Gibbs distributions on the Boolean hypercube are approximate
mixtures of product measures whose probability vectors are critical
points of an associated mean-field functional. This extends a previous
work by the first author. As an application, we demonstrate how this
framework helps characterize both Ising models satisfying a mean-field
condition and the conditional distributions which arise in the emerging
theory of nonlinear large deviations, both in the dense case and in
the polynomially-sparse case. 
\end{abstract}
\tableofcontents{}

\section{Introduction}

Let $n>0$ and let $f:\left\{ -1,1\right\} ^{n}\to\bR$ be a function.
A probability measure $\nu$ on $\left\{ -1,1\right\} ^{n}$ is called
a \emph{Gibbs distribution with Hamiltonian $f$} if for $X\sim\nu$,
\[
\pr{X=x}=\exp\left(f\left(x\right)\right)/Z,
\]
where $Z$ is a normalizing constant. We denote such a distribution
by $X_{n}^{f}$. Gibbs distributions are central to statistical physics,
and appear in applications in computer science, statistics, and economics.
However, many important Hamiltonians are far from being analytically
tractable.

One method to tackle the difficulties entrenched in such Hamiltonians
is via mean-field approximations. This method goes back to Curie and
Weiss and has long been widely used by physicists. More recently,
such approximations were established in rigor, see for example \cite{BasakMukherjee2017}.

For the case of Gibbs distributions on the Boolean hypercube, \cite{2016_eldan_gaussian_width}
showed that if the image of the gradient of the Hamiltonian $f$ has
small enough Gaussian-width and Lipschitz constants, then the partition
function can be approximated by applying the mean-field variant of
the Gibbs variational principle. Further, under the same conditions,
$X_{n}^{f}$ can be approximated by a mixture of product measures.
This improves an earlier result by Chatterjee and Dembo \cite{chatterjee_dembo_2014}
who consider a slightly different notion of complexity.

In this paper, we extend the framework introduced in \cite{2016_eldan_gaussian_width}
by showing that if the discrete gradient $\grad f$ also has a small
enough Lipschitz constant, then the product measures described above
are close to critical points of an associated variational functional
which corresponds to the so-called \emph{mean-field equations}. This
gives a more precise characterization of the mixture. 

An interesting feature of our framework is that it allows us to effectively
bypass the need to obtain an accurate approximation of the normalizing
constant in the route to understanding the Gibbs distribution. Even
though the approximations to the normalizing constant obtained by
the framework are far from sharp (they miss by a factor of $e^{o\left(n\right)}$
as seen in the examples in \cite{2016_eldan_gaussian_width}), our
results still manage to give information about the set where most
of the mass resides.

The following is an overview of our results. 
\begin{itemize}
	\item In \textbf{Theorem \ref{thm:main_theorem}}, we show that if the Hamiltonian
	$f$ has low complexity and satisfies a Lipschitz condition, the corresponding
	Gibbs distribution behaves like a mixture of densities of vectors
	whose entries are i.i.d Bernoulli random variables, and whose expectations
	$X$ satisfy
	\[
	\onenorm{X-\tanh\left(\grad f\left(X\right)\right)}=o\left(n\right),
	\]
	where the $\tanh$ is applied entrywise. 
	\item As an example of using this bound, we demonstrate in \textbf{Corollaries
		\ref{cor:general_ising} and \ref{cor:curie_wiess_ising}} that Ising
	models satisfying a mean-field assumption can be decomposed into product
	measures.
	\item \textbf{Theorem \ref{thm:composition_theorem} }concerns compositions:
	If a function $h:\bR\to\bR$ has small enough derivatives, then the
	function $h\circ f$ also satisfies Theorem \ref{thm:main_theorem}. 
	\item As an example of this composition, we demonstrate in \textbf{Theorem
		\ref{thm:large_deviations}} that the conditional distribution $\pr{Y=y\mid f\left(Y\right)\geq tn}$
	arising in large deviation theory can be approximated by a smoothed-cutoff
	distribution that can be decomposed into product measures, each satisfying
	an equation which arises from the Lagrange multiplier problem associated
	with the rate function. 
\end{itemize}
In the sequel work \cite{eldan_gross_exponential_random_graphs},
we apply Theorem \ref{thm:main_theorem} to exponential random graphs,
improving a previously known characterization.

\section{Background and notation}

We denote the Boolean hypercube by $\mathcal C_{n}=\left\{ -1,1\right\} ^{n}$
and the continuous hypercube by $\overline{\mathcal C_{n}}=\left[0,1\right]^{n}$.
The uniform measure on $\mathcal C_{n}$ is denoted by $\mu$. The space
of all product measures on $\mathcal C_{n}$ is denoted $\mathcal{PM}_{n}$.
For a vector $x\in\bR^{n}$, we denote its one-norm by 
\[
\onenorm x=\sum_{i=1}^{n}\abs{x_{i}}.
\]

\subsection{Two motivating examples of Hamiltonians}

\subsubsection{The Ising model\label{subsec:ising_intro_section}}

An Ising model on $n$ sites can be described as follows: Let $x\in\mathcal C_{n}$
represent $n$ interacting sites that can be in one of two states.
Let $A\in\bR^{n\times n}$ be a real symmetric matrix with $0$ on
the diagonal representing the intensity of interaction between the
sites, so that the interaction between site $i$ and site $j$ is
$A_{ij}$. Let $\mu\in\bR^{n}$ be a vector representing magnetic field
strengths, so that site $i$ feels a magnetic field $\mu_{i}$. The
Hamiltonian for the system is then defined as
\[
f\left(x\right)=\left\langle x,Ax\right\rangle +\left\langle \mu,x\right\rangle .
\]
If $\tr A^{2}=o\left(n\right)$, we say that the model satisfies the
\emph{mean-field} assumption \cite{BasakMukherjee2017}. We also assume
that both $\mu_{\max}$ and $\max_{i\in\left[n\right]}\sum_{j\in\left[n\right]}\abs{A_{ij}}$
are $O\left(1\right)$, which amounts to the force acting on a single
site being bounded.

\subsubsection{Nonlinear large deviations\label{subsec:large_deviations_intro_section}}

Let $f:\mathcal C_{n}\to\bR$ be a Hamiltonian. For $0\leq p\leq1$, define
$\mu_{p}$ to be the measure on $\mathcal C_{n}$ where every entry is
an i.i.d Bernoulli random variable with success probability $p$.
Let $t\in\bR$ be a real number. The two central questions in the field
of large deviation theory are:
\begin{enumerate}
	\item \label{enu:intro_ld_probability}For $Y\sim\mu_{p}$, what is the
	probability $\pr{f\left(Y\right)\geq tn}$?
	\item \label{enu:intro_ld_conditional}For $Y\sim\mu_{p}$, what is the
	conditional distribution $\pr{Y=y\mid f\left(Y\right)\geq tn}$?
\end{enumerate}
One line of approach to answering these questions is to approximate
$\pr{f\left(Y\right)\geq tn}$ and $\pr{Y=y\mid f\left(Y\right)\geq tn}$
by using Gibbs distributions. For example, observe that the conditional
distribution $\pr{Y=y\mid f\left(y\right)\geq tn}$ may be obtained
from a Gibbs distribution with a ``cutoff Hamiltonian'' $\tilde{f}$,
defined by

\begin{equation}
\tilde{f}\left(y\right)=\begin{cases}
\prod_{i=1}^{n}\log\left(\frac{1}{2}\left(1-y_{i}+2py_{i}\right)\right) & f\left(y\right)\geq tn\\
-\infty & f\left(y\right)<tn.
\end{cases}\label{eq:the_sharpest_cutoff}
\end{equation}
All $y$ with $f\left(y\right)\geq tn$ are thus weighted according
to $\mu_{p}$, and all $y$ with $f\left(y\right)<tn$ have probability
$0$. Unfortunately, $\tilde{f}$ is not smooth enough in order to
be applicable for the existing large deviation frameworks. However,
it is possible to get approximations of $X_{n}^{\tilde{f}}$ by using
Hamiltonians which approximate $\tilde{f}$. Such a ``smooth-cutoff''
Hamiltonian should give a large mass to ``good'' vectors $y$ such
that $f\left(y\right)\geq tn$ and a small mass to ``bad'' vectors
$y$ such that $f\left(y\right)<tn$. Both \cite{2016_eldan_gaussian_width}
and \cite{chatterjee_dembo_2014} follow this approach in order to
tackle item (\ref{enu:intro_ld_probability}).

\subsection{Boolean functions}
\begin{defn}[Discrete gradient, Lipschitz constant]
	\label{def:discrete_gradient}Let $f:\mathcal C_{n}\to\bR$ be
	a real function on the Boolean hypercube. The derivative of $f$ at
	coordinate $i$ is defined as
	\[
	\partial_{i}f\left(y\right)=\frac{1}{2}\left(f\left(y_{1},\ldots,y_{i-1},1,y_{+1},\ldots y_{n}\right)-f\left(y_{1},\ldots,y_{i-1},-1,y_{+1},\ldots y_{n}\right)\right).
	\]
	With this we define both the the discrete gradient:
	\[
	\grad f\left(y\right)=\left(\partial_{1}f\left(y\right),\ldots,\partial_{n}f\left(y\right)\right),
	\]
	and the Lipschitz constant of $f$: 
	\[
	\text{Lip}\left(f\right)=\max_{i\in\left[n\right],y\in\left\{ -1,1\right\} ^{n}}\abs{\partial_{i}f\left(y\right)}.
	\]
\end{defn}

Every Boolean function $f:\mathcal C_{n}\to\bR$ has a unique Fourier decomposition
into monomials \cite{oDonnell2014}:
\[
f\left(x\right)=\sum_{S\subseteq\left[n\right]}\hat{f}\left(S\right)\prod_{i\in S}x_{i}.
\]
This defines an extension of $f$ from the discrete hypercube $\mathcal C_{n}$
into the continuous hypercube $\overline{\mathcal C_{n}}=\left[-1,1\right]^{n}$
by computing the value of the polynomial $\sum_{S\subseteq\left[n\right]}\hat{f}\left(S\right)\prod_{i\in S}x_{i}$
for $x\in\overline{\mathcal C_{n}}$. It can be shown that this is the
same extension as the harmonic extension defined in \cite[Section 3.1.1]{2016_eldan_gaussian_width}.
By Fact 14 in \cite{2016_eldan_gaussian_width}, the extension of
$\partial_{i}f$ agrees with the $i$-th partial derivative (in the
real-differentiable sense) of the extension of $f$. Throughout this
text, we will always assume that $f$, and therefore $\grad f$ as
well, are extended to $\overline{\mathcal C_{n}}$.
\begin{defn}[Gaussian width, gradient complexity]
	The \emph{Gaussian-width} of a set $K\subseteq\bR^{n}$ is defined
	as
	\[
	\boldsymbol{\mathrm{GW}}\left(K\right)=\e\left[\sup_{x\in K}\left\langle x,\Gamma\right\rangle \right]
	\]
	where $\Gamma\sim N\left(0,\textrm{Id}\right)$ is a standard Gaussian
	vector in $\bR^{n}$. For a function $f:\mathcal C_{n}\to\bR$, the \emph{gradient
		complexity} of $f$ is defined as
	\[
	\mathcal D\left(f\right)=\boldsymbol{\mathrm{GW}}\left(\left\{ \grad f\left(y\right):y\in\mathcal C_{n}\right\} \union\left\{ 0\right\} \right).
	\]
	For a measure $\nu$ on $\mathcal C_{n}$, by slight abuse of notation,
	we define its complexity as
	\[
	\mathcal D\left(\nu\right)=\mathcal D\left(\log\frac{d\nu}{d\mu}\right).
	\]
\end{defn}

\subsection{Mixture models}
\begin{defn}[$\rho$-mixtures]
	For $z\in\left[-1,1\right]^{n}$, denote by $X\left(z\right)$ the
	unique random vector in $\mathcal C_{n}$ whose coordinates are independent
	and whose expectation is $\e X\left(z\right)=z$. Let $\rho$ be a
	measure on $\left[-1,1\right]^{n}$. We define the random vector $X\left(\rho\right)$
	by
	\begin{equation}
	\pr{X\left(\rho\right)=x}=\int\pr{X\left(z\right)=x}d\rho\left(z\right).\label{eq:rho_mixture}
	\end{equation}
\end{defn}

\begin{defn}[Approximate mixture decomposition]
	Let $\delta>0$ and let $\rho$ be a measure on $\left[-1,1\right]^{n}$.
	A random variable $X$ is called a \emph{$\left(\rho,\delta\right)$-mixture
	}if there exists a coupling between $X\left(\rho\right)$ and $X$
	such that 
	\[
	\e\onenorm{X\left(\rho\right)-X}\leq\delta n.
	\]
\end{defn}

A result of \cite{2016_eldan_gaussian_width} roughly states that
low complexity Gibbs distributions are $\left(\rho,\delta\right)$-mixtures
for $\delta=o\left(1\right)$ and where $\rho$ is such that most
of the entropy comes from the individual $X\left(z\right)$ rather
than from the mixture.
\begin{defn}[Wasserstein distance]
	For two distributions $\nu_{1}$ and $\nu_{2}$, the \emph{Wasserstein
		mass-transportation distance}, denoted $W_{1}$, is defined as 
	\[
	W_{1}\left(\nu_{1},\nu_{2}\right)=\inf_{\underset{X\sim\nu_{1},Y\sim\nu_{2}}{\left(X,Y\right)\,s.t}}\frac{1}{2}\e\onenorm{X-Y},
	\]
	where the infimum is taken over all joint distributions whose marginals
	have the laws $\nu_{1}$ and $\nu_{2}$ respectively.
\end{defn}

\begin{defn}[Tilt of a distribution]
	For a vector $\theta\in\bR^{n}$, the \emph{tilt} $\tau_{\theta}\nu$
	of the distribution $\nu$ is a distribution defined by
\end{defn}

\[
\frac{d\left(\tau_{\theta}\nu\right)}{d\nu}\left(y\right)=\frac{e^{\left\langle \theta,y\right\rangle }}{\int_{\mathcal C_{n}}e^{\left\langle \theta,z\right\rangle }d\nu}.
\]

With the notion of $\rho$-mixture and tilt, we define what it means
for a random variable to break up into small tilts: 
\begin{defn}[Tilt decomposition]
	\label{def:tilt_decomposition}Let $\delta,\eps>0$ and let $\rho$
	be a measure on $\left[-1,1\right]^{n}$. A random variable $X$ with
	distribution $\nu$ is called a \emph{$\left(\rho,\delta,\eps\right)$-tilt-mixture}
	if there exists a probability measure $m$ on $\bR^{n}$ supported
	on $B\left(0,\eps\sqrt{n}\right)\intersect\left[-\frac{1}{4},\frac{1}{4}\right]^{n}$
	such that:
\end{defn}

\begin{enumerate}
	\item \label{enu:tilt_decomposition_1_decomposition}For every $\varphi:\mathcal C_{n}\to\bR$,
	\[
	\int_{\mathcal C_{n}}\varphi d\nu=\int_{\bR^{n}}\left(\int_{\mathcal C_{n}}\varphi d\left(\tau_{\theta}\nu\right)\right)dm\left(\theta\right).
	\]
	\item \label{enu:tilt_decomposition_2_high_prob}For all but a $\delta$-portion
	of the measure $m$, the tilt $\tau_{\theta}\nu$ is $\delta n$-close
	to a product measure in Wasserstein distance:
	\[
	m\left(\left\{ \theta\in\bR^{n}:\exists\xi\in\mathcal{PM}_{n}\,s.t\,W_{1}\left(\tau_{\theta}\nu,\xi\right)\leq\delta n\right\} \right)>1-\delta.
	\]
	\item \label{enu:tilt_decomposition_3_push_forward}The measure $\rho$
	is the push-forward of the measure $m$ under the map $\theta\mapsto\e_{X\sim\tau_{\theta}\nu}\left[X\right]$.
\end{enumerate}
\begin{fact}
	Every $\left(\rho,\delta,\eps\right)$-tilt-mixture is also a $\left(\rho,4\delta\right)$-mixture.
\end{fact}

\begin{proof}
	Define $\Theta=\left\{ \theta\in\bR^{n}:\exists\xi\in\mathcal{PM}_{n}\,s.t\,W_{1}\left(\tau_{\theta}\nu,\xi\right)\leq\delta n\right\} $,
	and denote the distribution of $X$ and of $X\left(\rho\right)$ by
	$\nu$ and $\sigma$ respectively. Using item \ref{enu:tilt_decomposition_1_decomposition}
	in the definition of a tilt-mixture, we have
	
	\begin{align*}
	W_{1}\left(\nu,\sigma\right) & \leq\int_{\bR^{n}}W_{1}\left(\xi_{\theta},\tau_{\theta}\nu\right)dm\left(\theta\right)\\
	& \leq\int_{\Theta}W_{1}\left(\xi_{\theta},\tau_{\theta}\nu\right)dm\left(\theta\right)+m\left(\left[-1/4,1/4\right]^{n}\backslash\Theta\right)n.
	\end{align*}
	By item \ref{enu:tilt_decomposition_2_high_prob} in the definition
	of a tilt-mixture, there exists a coupling between $X$ and $X\left(\rho\right)$
	such that each term on the right hand side is bounded by $\delta n$.
	This gives a $4\delta$ bound on the expectation $\e\onenorm{X-X\left(\rho\right)}$.
\end{proof}
A tilt-mixture decomposition provides more information than general$\rho$-mixtures:
It tells us something about the structure of the elements of the mixture,
with the parameter $\eps$ in Definition \ref{def:tilt_decomposition}
bounding the support of the tilts to a ball of radius $\eps\sqrt{n}$.
Some of our results will rely on the existence of tilt decompositions
with small $\eps$.

\section{Results}

Our main technical contribution is a characterization of the measure
$\rho$ described above: With high probability with respect to $\rho$,
the vector $z$ in equation (\ref{eq:rho_mixture}) is nearly a critical
point of a certain functional associated with $f$.
\begin{thm}[Main Structural Theorem]
	\label{thm:main_theorem}Let $n>0$, let $f:\mathcal C_{n}\to\bR$ be
	a function and denote 
	\begin{align}
	D & =\mathcal D\left(f\right)\label{eq:bound_on_D}\\
	L_{1} & =\max\left\{ 1,\text{Lip}\left(f\right)\right\} \label{eq:bound_on_lipschitz}\\
	L_{2} & =\max\left\{ 1,\max_{x\neq y\in\mathcal C_{n}}\frac{\onenorm{\grad f\left(x\right)-\grad f\left(y\right)}}{\onenorm{x-y}}\right\} .\label{eq:bound_on_grad_lipschitz}
	\end{align}
	Denote by $\mathcal X_{f}$ the set 
	\begin{equation}
	\mathcal X_{f}=\left\{ X\in\overline{\mathcal C_{n}}:\onenorm{X-\tanh\left(\grad f\left(X\right)\right)}\leq5000L_{1}L_{2}^{3/4}D^{1/4}n^{3/4}\right\} ,\label{eq:main_theorem_inequality}
	\end{equation}
	where $\grad f\left(X\right)$ is calculated by harmonically extending
	$\grad f$ to $\overline{\mathcal C_{n}}$, and with the $\tanh$ applied
	entrywise to the entries of $\grad f\left(X\right)$. Then $X_{n}^{f}$
	is a $\left(\rho,\frac{3D^{1/4}}{n^{1/4}},L_{2}^{3/4}\frac{D^{1/4}}{n^{1/4}}\right)$-tilt-mixture
	such that 
	\begin{equation}
	\rho\left(\mathcal X_{f}\right)\geq1-\frac{3D^{1/4}}{n^{1/4}}.\label{eq:main_theorem_probability_statement}
	\end{equation}
	In particular, if $D=o\left(n\right)$, then $X_{n}^{f}$ is a $\left(\rho,o\left(1\right)\right)$-mixture
	with $\rho\left(\mathcal X_{f}\right)=1-o\left(1\right)$.
\end{thm}

In other words, almost all the mass of the mixture resides on random
vectors $X$ which almost satisfy the fixed point equation

\begin{equation}
X=\tanh\left(\grad f\left(X\right)\right).\label{eq:fixed_point_equation}
\end{equation}
\begin{rem}
	One can check that the solutions of the fixed point equation are exactly
	the critical points of the functional $f\left(X\right)+H\left(X\right)$
	where $H\left(X\right)=\sum_{i<,j}X_{ij}\log X_{ij}+\left(1-X_{ij}\right)\log\left(1-X_{ij}\right)$
	is the entropy of $X$. This is a variant of the functional that arises
	in the variational problem in \cite{chatterjee_diaconis_2013}.
\end{rem}

\begin{rem}
	The following is an example application of Theorem \ref{thm:main_theorem}
	to Ising models, to be compared with the main result of \cite{BasakMukherjee2017}.
\end{rem}

\begin{cor}[Ising models]
	\label{cor:general_ising}Let $f$ be an Ising model Hamiltonian
	as described in Section \ref{subsec:ising_intro_section}, with interaction
	matrix $A\in\bR^{n\times n}$ and a magnetic moment vector $\mu\in\bR^{n}$.
	Denote
	\[
	\mathcal X_{f}=\left\{ X\in\overline{\mathcal C_{n}}:\onenorm{X-\tanh\left(AX+\mu\right)}\leq5000L_{1}L_{2}^{3/4}D^{1/4}n^{3/4}\right\} ,
	\]
	where 
	\begin{align*}
	D & =\sqrt{n\tr A^{2}}+\sqrt{n}\mu_{\max}\\
	L_{1} & =\max\left\{ 1,\mu_{\max}+\max_{i\in\left[n\right]}\sum_{j\in\left[n\right]}\abs{A_{ij}}\right\} \\
	L_{2} & =\max\left\{ 1,\max_{i\in\left[n\right]}\sum_{j\in\left[n\right]}\abs{A_{ij}}\right\} .
	\end{align*}
	Then $X_{n}^{f}$ is a $\left(\rho,\frac{3D^{1/4}}{n^{1/4}},L_{2}^{3/4}\frac{D^{1/4}}{n^{1/4}}\right)$-tilt-mixture
	such that 
	\[
	\rho\left(\mathcal X_{f}\right)\geq1-\frac{3D^{1/4}}{n^{1/4}}.
	\]
	In particular, if $L_{1}=O\left(1\right)$ and $\tr\left(A^{2}\right)=o\left(n\right)$
	(the ``mean-field assumption''), then $\mathcal X_{f}$ is $\left(\rho,o\left(1\right)\right)$-mixture
	with $\rho\left(\mathcal X_{f}\right)=1-o\left(1\right)$.
\end{cor}

The simplest example of an Ising model is the Curie-Weiss ferromagnet,
for which we can use our framework as a toy example and rederive well-known
properties about its distribution.
\begin{cor}
	\label{cor:curie_wiess_ising}Let $\beta>0$ and let $f:\mathcal C_{n}\to\bR$
	be the Curie-Weiss Hamiltonian, $f\left(x\right)=\frac{\beta}{n}\sum_{i\neq j}x_{i}x_{j}.$
	Denote
	\[
	\mathcal X_{f}=\left\{ X\in\overline{\mathcal C_{n}}:\onenorm{X-\tanh\left(\frac{\beta\boldsymbol{J}}{n}X\right)}\leq5001\left(1+\beta\right)^{2}n^{7/8}\right\} ,
	\]
	where $\boldsymbol{J}$ is the $n\times n$ all-$1$ matrix. Then
	$X_{n}^{f}$ is a $\left(\rho,3n^{-1/8},3n^{-1/8}\right)$-tilt-mixture,
	and $\rho\left(\mathcal X_{f}\right)\geq1-3n^{-1/8}$. Further, if $\beta<1$,
	then every $X\in\mathcal X_{f}$ satisfies 
	\[
	\onenorm X\leq5001\frac{\left(1+\beta\right)^{2}}{1-\beta}n^{7/8}.
	\]
\end{cor}

For a more detailed application of Theorem \ref{thm:main_theorem}
for the case of exponential random graphs, see \cite{eldan_gross_exponential_random_graphs}.

The following theorem finds sufficient conditions under which composing
$f$ with a real-valued function produces a Hamiltonian with a$\rho$-mixture
approximation:
\begin{thm}[Composition Theorem]
	\label{thm:composition_theorem}Let $h:\bR \to \bR$ be a twice differentiable
	function satisfying
	
	\begin{align*}
	\abs{h'\left(x\right)} & <B_{1}\,\,\,\,\,\,\,\all x\in\bR\\
	\abs{h''\left(x\right)} & <B_{2}\,\,\,\,\,\,\,\all x\in\bR.
	\end{align*}
	Let $f:\mathcal C_{n}\to\bR$ be a function with parameters $D$, $L_{1}$,
	and $L_{2}$ as described in Theorem \ref{thm:main_theorem}. Denote
	by $\tilde{D}$, $\tilde{L}_{1}$, $\tilde{L}_{2}$ and $\tilde{L}_{3}$
	the real numbers 
	\begin{align*}
	\tilde{D} & =B_{1}D+B_{2}L_{1}^{2}n\\
	\tilde{L}_{1} & =\max\left\{ 1,B_{1}L_{1}\right\} \\
	\tilde{L}_{2} & =\max\left\{ 1,B_{1}L_{2}+3B_{2}L_{1}^{2}n\right\} \\
	\tilde{L}_{3} & =2B_{2}L_{1}^{2}n^{3/2}
	\end{align*}
	and denote by $\mathcal{\tilde{X}}_{h\circ f}$ the set 
	\begin{equation}
	\mathcal{\tilde{X}}_{h\circ f}=\left\{ X\in\overline{\mathcal C_{n}}:\onenorm{X-\tanh\left(h'\left(f\left(X\right)\right)\grad f\left(X\right)\right)}\leq5000\tilde{L}_{1}\tilde{L}_{2}^{3/4}\tilde{D}^{1/4}n^{3/4}+\tilde{L}_{3}\right\} ,\label{eq:composition_theorem_inequality}
	\end{equation}
	where $\grad f\left(X\right)$ is calculated by harmonically extending
	$\grad f$ to $\overline{\mathcal C_{n}}$, and with the $\tanh$ applied
	entrywise to the entries of $\grad f\left(X\right)$. Then $X_{n}^{h\circ f}$
	is a $\left(\rho,\frac{3\tilde{D}^{1/4}}{n^{1/4}},\tilde{L}_{2}^{3/4}\frac{\tilde{D}^{1/4}}{n^{1/4}}\right)$-tilt-mixture
	such that 
	\[
	\rho\left(\mathcal X_{h\circ f}\right)\geq1-\frac{3\tilde{D}^{1/4}}{n^{1/4}}.
	\]
\end{thm}

\begin{rem}
	\label{rem:a_slight_difference}Theorem \ref{thm:composition_theorem}
	bounds the norm $\left\Vert X-\tanh\left(h'\left(f\left(X\right)\right)\grad f\left(X\right)\right)\right\Vert _{1}$
	rather than $\left\Vert X-\tanh\left(\grad\left(h\circ f\right)\left(X\right)\right)\right\Vert _{1}$
	(which is the analogue of the quantity arising in the main Theorem
	\ref{thm:main_theorem}). This is a matter of practicality: For many
	known Hamiltonians $f$ it is easy to compute $\grad f$ and its extension
	to $\overline{\mathcal C_{n}}$, but it is not straightforward to compute
	$\grad\left(h\circ f\right)\left(X\right)$ and its extension to $\overline{\mathcal C_{n}}$
	for arbitrary $h$. In these cases, calculating $h'\left(f\left(X\right)\right)\grad f\left(X\right)$
	is a much simpler task. Further, as will be shown in Lemma \ref{lem:chain_rule},
	the two quantities $h'\left(f\left(X\right)\right)\grad f\left(X\right)$
	and $\grad\left(h\circ f\right)\left(X\right)$ are close to each
	other.
\end{rem}

As an example application of Theorem \ref{thm:composition_theorem},
we show that the conditional distribution $\pr{Y=y\mid f\left(y\right)\geq tn}$
described in item (\ref{enu:intro_ld_conditional}) in Section \ref{subsec:large_deviations_intro_section}
can be approximated by a ``smoothed-out'' distribution, which gives
equal mass to vectors $y$ satisfying $f\left(y\right)\geq nt$ and
no mass to vectors $y$ satisfying $f\left(y\right)<\left(t-\delta\right)n$.
This ``smoothed-out'' distribution is obtained from a ``smoothed-cutoff''
approximation to the $\tilde{f}$ described in Section \ref{subsec:large_deviations_intro_section}.
Our framework can be applied to this ``smoothed-cutoff'' function,
yielding an equation corresponding to the Lagrange multiplier problem
associated with the rate function.
\begin{thm}[Large deviations]
	\label{thm:large_deviations}Let $t>0$. Let $f:\mathcal C_{n}\to\bR$
	be a Hamiltonian with parameters $D$, $L_{1}$ and $L_{2}$ as described
	in Theorem \ref{thm:main_theorem}, and assume that there exists $z\in\mathcal C_{n}$
	such that $f\left(z\right)\geq tn$. Let $\delta>0$. There exists
	a monotone function $h:\bR\to\bR$, such that for $\varphi=h\circ f$,
	we have that $\varphi\left(y\right)=0$ if $f\left(y\right)<\left(t-2\delta\right)n$,
	$\varphi\left(y\right)=1$ if $f\left(y\right)\geq tn$ and such that
	the following holds. Denote by $\sigma$ the measure defined by 
	\[
	d\sigma=\frac{\varphi d\mu}{\int_{\mathcal C_{n}}\varphi d\mu},
	\]
	and let $X_{\varphi}$ be a random variable whose law is $\sigma$.
	Denote
	\begin{align*}
	\mathcal X_{g} & =\left\{ X\in\overline{\mathcal C_{n}}:\exists\lambda\in\bR\,s.t.\,\onenorm{X-\tanh\left(\lambda\grad f\left(X\right)\right)}\leq5000\tilde{L}_{1}\tilde{L}_{2}^{3/4}\tilde{D}^{1/4}n^{3/4}+\tilde{L}_{3}\right.\\
	& \left.\,\,\,\,\,\,\,\,\,\,\,\,\,\,\,\,\,\,\,\,\,\,\,\,\,\,\,\,\,\,\,\,\,\,\,\,\,\,\,\,\,\,\,\,\,\,\,\,\,\text{and }f\left(X\right)\in\left[\left(t-6\delta\right)n,tn\right]\right\} 
	\end{align*}
	where
	\begin{align*}
	\tilde{D} & =\frac{2}{\delta}D+\frac{2}{\delta^{2}}L_{1}^{2}\\
	\tilde{L}_{1} & =\max\left\{ 1,\frac{2}{\delta}L_{1}\right\} \\
	\tilde{L}_{2} & =\max\left\{ 1,\frac{2}{\delta}L_{2}^{2}+3\frac{2}{\delta^{2}}L_{1}^{2}\right\} \\
	\tilde{L}_{3} & =2\frac{2}{\delta^{2}}L_{1}^{2}n^{1/2}.
	\end{align*}
	Then $X_{\varphi}$ is a $\left(\rho,80\frac{\tilde{D}^{1/4}}{n^{1/4}}+8\cdot2^{-n}\right)$-mixture
	such that 
	\begin{equation}
	\rho\left(\mathcal X_{g}\right)\geq1-\frac{165L_{1}\tilde{D}^{1/4}}{n^{1/4}2\delta}\left(1-\frac{L_{1}}{2\delta\sqrt{n}}-2^{-n}\right)^{-1}.\label{eq:large_deviation_probability_statement}
	\end{equation}
\end{thm}

Note that the expression $X-\tanh\left(\lambda\grad f\left(X\right)\right)$
in the definition of the set $\mathcal X_{g}$ is closely related to the
rate function: Consider the variational problem

\begin{flalign*}
& \text{minimize\,\ensuremath{H\left(Y\right)}}\\
& \text{subject to \ensuremath{\e f\left(Y\right)\geq tn}}
\end{flalign*}
where $Y$ is a random vector in $\mathcal C_{n}$ whose entries are independent.
By monotonicity, the minimum is attained on the boundary of the constraint.
Denoting $\e Y=y$ and using the method of Lagrange multipliers, we
obtain the equations
\begin{align}
\grad_{y}H\left(Y\right) & =\lambda\grad f\left(y\right)\label{eq:lagrange_mult_1}\\
f\left(y\right) & =tn.\nonumber 
\end{align}
Applying the fact that $\grad_{y}H\left(Y\right)=\tanh^{-1}\left(y\right)$
on equation (\ref{eq:lagrange_mult_1}) gives exactly the equation
$X-\tanh\left(\lambda\grad f\left(X\right)\right)=0$.
\begin{example}[Large deviations for triangle counts]
	\label{exa:triangles}Let $N>0$ be an integer representing the number
	of vertices of a graph, and let $n={N \choose 2}$ be the number of
	possible edges in the graph. We treat each vector $v\in\mathcal C_{n}$
	as a simple graph, with $v_{e}=1$ if and only if the edge $e$ appears
	in the graph. This in turns gives an adjacency matrix $\left(x_{ij}\right)_{i,j=1}^{N}$
	with $x_{ij}=1$ if and only if $v_{\left\{ ij\right\} }=1$. In this
	setting, let $f$ be a triangle-counting function,
	\[
	f\left(x\right)=\frac{\beta}{N}\sum_{i\neq j\neq k}x_{ij}x_{jk}x_{ki}
	\]
	for some real $\beta$. It is shown in \cite{2016_eldan_gaussian_width}
	that $\mathcal D\left(f\right)$ is $O\left(n^{3/4}\right)$ and in \cite{eldan_gross_exponential_random_graphs}
	that $L_{1}$ and $L_{2}$ are bounded by $200\abs{\beta}$. Thus
	we can apply Theorem \ref{thm:large_deviations} to $f$, concluding
	that for a fixed $t>0$ there exists some $\delta=o\left(1\right)$
	and a smoothed cutoff function $h$ with $h\left(x\right)=1$ for
	$x>tn$ and $h\left(x\right)=0$ for $x<(t-\delta)n$ and such that
	the random graph $G$ whose density is proportional to $h\circ f$
	is a $\left(\rho,o\left(1\right)\right)$-mixture such that $\rho\left(\mathcal X_{g}\right)=1-o\left(1\right)$,
	where
	\begin{align*}
	\mathcal X_{g} & =\left\{ X\in\overline{\mathcal C_{n}}:\exists\lambda\in\bR\,s.t.\,\onenorm{X-\tanh\left(\lambda X^{2}\right)}\leq\eps n\right.\\
	& \left.\,\,\,\,\,\,\,\,\,\,\,\,\,\,\,\,\,\,\,\,\,\,\,\,\,\,\,\,\,\,\,\,\,\,\,\,\,\,\,\,\,\,\,\,\,\,\,\,\,\text{and }f\left(X\right)\in\left[\left(t-6\delta\right)n,tn\right]\right\} 
	\end{align*}
	for some $\eps=o\left(1\right)$. Here $X\in\overline{\mathcal{C}_{n}}$
	is treated as an $n\times n$ symmetric matrix with zeros on the diagonal,
	and we understand the expression $X^{2}$ as the usual matrix multiplication,
	with zeros on the diagonal as well. We conjecture that all of the
	points of the set $\mathcal X_{g}$ are close to the solutions obtained
	by Lubetzky and Zhao in \cite{lubetzky_zhao_sparse}. 
	
	Our results extend to triangle counts on sparse graphs as well. In
	this case, expected value of $f$ is of order $np^{3}$, which decays
	to $0$ as $p\to0$. We should therefore take both $t$ to be proportional
	to $p^{3}$ and $\delta$ to be $o\left(p^{3}\right)$. Since the
	bound on the vectors in $\mathcal{X}_{g}$ in Theorem \ref{thm:large_deviations}
	is polynomial in $\delta$, we can consider large deviations for graphs
	whose edge probabilities are proportional to $p\sim n^{-c}$ for some
	constant $c$ (for example, if we wish $\eps$ to be of order $p$,
	we can take $p\sim n^{-1/160}$).
\end{example}

The rest of this paper is organized as follows. Theorem \ref{thm:main_theorem}
is proved in Section \ref{sec:proof-of-main}, while Theorem \ref{thm:composition_theorem}
is proved in Section \ref{sec:composition_theorem}. Corollaries \ref{cor:general_ising}
and \ref{cor:curie_wiess_ising} are proved in Section \ref{subsec:ising_proof}
and \ref{thm:large_deviations} is proved in Section \ref{subsec:large_deviations_proof}.

\section{Proof of main theorem\label{sec:proof-of-main}}

\subsection{Notation and review}

We use the notation from \cite{2016_eldan_gaussian_width}, and rely
on the proofs therein. Here is a brief review of the required terms
and bounds. 

For a probability measure $\nu$ on $\mathcal C_{n}$, we define $f_{\nu}=\log\left(d\nu/d\mu\right)$,
so that the Gibbs distribution with Hamiltonian $f_{\nu}$ is exactly
$\nu$. For every distribution $\nu$ on the hypercube (exponential
or otherwise), we define 
\[
\mathcal H\left(\nu\right)=\int_{\mathcal C_{n}}\tanh\left(\grad f_{\nu}\left(y\right)\right)^{\otimes2}d\nu-\left(\int_{\mathcal C_{n}}\tanh\left(\grad f_{\nu}\left(y\right)\right)d\nu\right)^{\otimes2},
\]
which should be thought of as the covariance matrix of the random
variable $\nabla f_{\nu}\left(X\right)$ with $X\sim\nu$. We will
use the following three results from \cite{2016_eldan_gaussian_width}. 
\begin{prop}[Proposition 17 in \cite{2016_eldan_gaussian_width}]
	\label{prop:eldan_prop_17}Let $\tilde{\nu}$ be a probability distribution
	on $\mathcal C_{n}$. Then there exists a product measure $\xi=\xi\left(\tilde{\nu}\right)$
	such that 
	\begin{equation}
	W_{1}\left(\tilde{\nu},\xi\right)\leq\sqrt{n\tr\left(\mathcal H\left(\tilde{\nu}\right)\right)}.\label{eq:wassertein}
	\end{equation}
	Moreover, one may take $\xi$ to be the unique product measure whose
	center of mass lies at $\int_{\mathcal C_{n}}\tanh\left(\grad f_{\tilde{\nu}}\left(y\right)\right)d\tilde{\nu}\left(y\right)$
	where the $\tanh$ is applied entrywise. 
\end{prop}

\begin{prop}[Proposition 18 together with Lemma 16 in \cite{2016_eldan_gaussian_width}]
	\label{prop:eldan_prop_18}Define $D=\mathcal D\left(f_{\nu}\right)$.
	Let $\eps\in\left(0,1/4\sqrt{\log\left(4n/D\right)}\right)$. Let
	$\nu$ be a probability measure on $\mathcal C_{n}$ and define $f=\log\frac{d\nu}{d\mu}$.
	Then there exists a measure $m$ on $B\left(0,\eps\sqrt{n}\right)\intersect\left[-1/4,1/4\right]^{n}$,
	such that $\nu$ admits the decomposition 
	\begin{equation}
	\int_{\mathcal C_{n}}\varphi d\nu=\int_{B\left(0,\eps\sqrt{n}\right)}\left(\int_{\mathcal C_{n}}\varphi d\tau_{\theta}\left(\nu\right)\right)dm\left(\theta\right)\label{eq:decomposition}
	\end{equation}
	for every test function $\varphi:\mathcal C_{n}\to\bR$, and which satisfies 
	
	\begin{equation}
	m\left(\theta:\,\tr\left(\mathcal H\left(\tau_{\theta}\nu\right)\right)\leq256\frac{n^{1/3}D^{2/3}}{\eps^{2/3}}\right)\geq1-\frac{3D^{1/3}}{n^{1/3}\eps^{1/3}}.\label{eq:trace_h}
	\end{equation}
\end{prop}

\begin{lem}[Lemma 24 in \cite{2016_eldan_gaussian_width}]
	\label{lem:fix_that_theta}Let $\theta\in\bR^{n}$ and let $\nu,\tilde{\nu}$
	be probability measures on $\mathcal C_{n}$. Define 
	\[
	A=\int_{\mathcal C_{n}}\tanh\left(\grad f_{\nu}\left(y\right)\right)^{\otimes2}d\tilde{\nu}-\left(\int_{\mathcal C_{n}}\tanh\left(\grad f_{\nu}\left(y\right)\right)d\tilde{\nu}\right)^{\otimes2}
	\]
	and
	
	\[
	B=\int_{\mathcal C_{n}}\tanh\left(\grad f_{\tau_{\theta}\nu}\left(y\right)\right)^{\otimes2}d\tilde{\nu}-\left(\int_{\mathcal C_{n}}\tanh\left(\grad f_{\tau_{\theta}\nu}\left(y\right)\right)d\tilde{\nu}\right)^{\otimes2}.
	\]
	Then 
	\[
	e^{-4\norm{\theta}_{\infty}}\tr B\leq\tr A\leq e^{4\norm{\theta}_{\infty}}\tr B.
	\]
\end{lem}

We can now describe the general plan of our proof. Fix $\eps>0$,
and let $m$ be the measure obtained from Proposition \ref{prop:eldan_prop_18}.
Denote by $\Theta$ the set 
\begin{equation}
\Theta=\left\{ \theta\in\bR^{n}:\,\tr\left(\mathcal H\left(\tau_{\theta}\nu\right)\right)\leq256\frac{n^{1/3}D^{2/3}}{\eps^{2/3}}\right\} .\label{eq:trace_in_intution}
\end{equation}
For every $\theta\in\bR^{n}$, denote by $\xi_{\theta}$ the unique
product measure with the same marginals as $\tau_{\theta}\nu$, and
by $A\left(\theta\right)$ the vector 
\[
A\left(\theta\right)=\mathrm{\e}_{X\sim\tau_{\theta}\nu}\left[X\right].
\]
Denote by $\rho$ the push-forward of the measure $m$ under the map
$\theta\mapsto A\left(\theta\right)$ and define
\[
\mathcal X=\left\{ A\left(\theta\right);\theta\in\Theta\right\} .
\]
In order to prove Theorem \ref{thm:main_theorem}, all we have to
do is that show that for each $\theta\in\Theta$, the corresponding
$A\left(\theta\right)$ is close in the one-norm to $\tanh\left(\grad f\left(A\left(\theta\right)\right)\right)$;
this will show equation (\ref{eq:main_theorem_probability_statement}).
In other words, we need the following proposition:
\begin{prop}
	\label{prop:main_proposition}Let $\theta\in\Theta$ and let $Y\sim\xi_{\theta}$.
	Then for every $\eps>0$, 
\end{prop}

\begin{align*}
\onenorm{\tanh\left(\grad f\left(\e Y\right)\right)-\e Y} & \leq41L_{1}\left(112L_{2}\frac{n^{2/3}D^{1/3}}{\eps^{1/3}}+\eps n\right).
\end{align*}
Relying on the above, we can prove of Theorem \ref{thm:main_theorem}.
\begin{proof}[Proof of Theorem \ref{thm:main_theorem}]
	Define the measure $\rho$ and the set $\mathcal X$ as above. Set $\eps=\frac{D^{1/4}L_{2}^{3/4}}{n^{1/4}}$.
	Items (\ref{enu:tilt_decomposition_1_decomposition})-(\ref{enu:tilt_decomposition_3_push_forward})
	in definition \ref{def:tilt_decomposition} follow immediately from
	Proposition \ref{prop:eldan_prop_17} and \ref{prop:eldan_prop_18}
	by choice of $\eps$, $\delta$ and $\rho$. By Proposition \ref{prop:main_proposition}
	for all $\theta\in\Theta$, we have
	\begin{align*}
	\onenorm{\tanh\left(\grad f\left(\e Y\right)\right)-\e Y} & \leq41L_{1}\left(113L_{2}^{3/4}D^{1/4}n^{3/4}\right)\\
	& \leq5000L_{1}L_{2}^{3/4}D^{1/4}n^{3/4}.
	\end{align*}
	This implies that $\mathcal X\subseteq\mathcal X_{f}$, and together with
	Proposition \ref{prop:eldan_prop_18} and by choice of $\eps$, this
	shows that $\rho\left(\mathcal X_{f}\right)\geq1-\frac{3D^{1/4}}{n^{1/4}}$,
	satisfying equation (\ref{eq:main_theorem_probability_statement}).
\end{proof}
The rest of this section is devoted to proving Proposition \ref{prop:main_proposition}.

\subsection{Approximate fixed point}

Let $\theta\in\Theta$ be a tilt and let $\xi_{\theta}$ be the product
measure whose center of mass lies at $\int_{\mathcal C_{n}}\tanh\left(\grad f_{\tau_{\theta}\nu}\left(y\right)\right)d\tau_{\theta}\nu\left(y\right)$.
Throughout the proof we will assume $X\sim\tau_{\theta}\nu$ and $Y\sim\xi_{\theta}$.
A direct calculation shows that under this notation, $\e Y=\e\tanh\left(\grad f\left(X\right)+\theta\right)$:
\begin{align*}
\e Y & =\int_{\mathcal C_{n}}\tanh\left(\grad f_{\tau_{\theta}\nu}\left(y\right)\right)d\tau_{\theta}\nu\left(y\right)\\
& =\int_{\mathcal C_{n}}\tanh\left(\grad\left(\log\left(\frac{d\tau_{\theta}\nu}{d\nu}\right)+\log\left(\frac{d\nu}{d\mu}\right)\right)\left(y\right)\right)d\tau_{\theta}\nu\left(y\right)\\
& =\int_{\mathcal C_{n}}\tanh\left(\theta+\grad f_{\nu}\left(y\right)\right)d\tau_{\theta}\nu\left(y\right)\\
& =\e\tanh\left(\grad f\left(X\right)+\theta\right).
\end{align*}
This gives 
\begin{align}
\onenorm{\e Y-\e\tanh\left(\grad f\left(X\right)\right)} & \leq\onenorm{\e\tanh\left(\grad f\left(X\right)\right)-\e\tanh\left(\grad f\left(X\right)+\theta\right)}\nonumber \\
& \leq\onenorm{\theta}\nonumber \\
\left(\twonorm{\theta}\leq\eps\sqrt{n}\right) & \leq\eps n.\label{eq:component_1}
\end{align}
where in the second inequality we use the fact that
\begin{equation}
\abs{\tanh\left(x\right)-\tanh\left(y\right)}\leq\abs{x-y}.\label{eq:tanh_is_contracting}
\end{equation}
\begin{prop}
	\label{prop:part_one}Let $Y\sim\xi_{\theta}$. Then 
	\[
	\e\onenorm{\tanh\left(\grad f\left(Y\right)\right)-\e Y}\leq64L_{2}\frac{n^{2/3}D^{1/3}}{\eps^{1/3}}+\eps n.
	\]
\end{prop}

\begin{proof}
	For $X\sim\tau_{\theta}\nu$, consider the random variable $Z=\norm{\tanh\left(\grad f\left(X\right)\right)-\e\tanh\left(\grad f\left(X\right)\right)}_{2}^{2}$.
	A short calculation shows that the expectation of $Z$ is roughly
	$\tr\mathcal H\left(\tau_{\theta}\nu\right)$:
	\begin{align*}
	\e Z & =\e\norm{\tanh\left(\grad f\left(X\right)\right)-\e\tanh\left(\grad f\left(X\right)\right)}_{2}^{2}\\
	& =\sum_{i=1}^{n}\e\left[\tanh\left(\grad f\left(X\right)\right)_{i}^{2}\right]-\sum_{i=1}^{n}\left(\e\tanh\left(\grad f\left(X\right)\right)_{i}\right)^{2}\\
	& \leq3\left(\sum_{i=1}^{n}\e\left[\tanh\left(\grad f\left(X\right)+\theta\right)_{i}^{2}\right]-\sum_{i=1}^{n}\left(\e\tanh\left(\grad f\left(X\right)+\theta\right)_{i}\right)^{2}\right)\\
	& =3\tr\left(\mathcal H\left(\tau_{\theta}\nu\right)\right)
	\end{align*}
	where the inequality is by Lemma \ref{lem:fix_that_theta} with $\nu$
	and $\tilde{\nu}=\tau_{\theta}\nu$ and the fact that $\norm{\theta}_{\infty}\leq1/4$.
	Thus by equation (\ref{eq:trace_h}),
	\begin{align*}
	\e\norm{\tanh\left(\grad f\left(X\right)\right)-\e\tanh\left(\grad f\left(X\right)\right)}_{2}^{2} & \leq3\cdot256\frac{n^{1/3}D^{2/3}}{\eps^{2/3}},
	\end{align*}
	and together with the Cauchy-Schwarz inequality, we have that
	\begin{align}
	\e\onenorm{\tanh\left(\grad f\left(X\right)\right)-\e\tanh\left(\grad f\left(X\right)\right)} & \leq\sqrt{n}\e\twonorm{\tanh\left(\grad f\left(X\right)\right)-\e\tanh\left(\grad f\left(X\right)\right)}\nonumber \\
	& \leq32\frac{n^{2/3}D^{1/3}}{\eps^{1/3}}.\label{eq:tanhx_is_close_to_etanhx}
	\end{align}
	By Proposition \ref{prop:eldan_prop_17}, there exists a coupling
	between $\tau_{\theta}\nu$ and $\xi_{\theta}$ such that 
	\begin{align*}
	\e\onenorm{X-Y} & \leq2\sqrt{n\tr\mathcal H\left(\tau_{\theta}\nu\right)}\\
	\left(\text{by equation (\ref{eq:trace_h})}\right) & \leq32\frac{n^{2/3}D^{1/3}}{\eps^{1/3}}.
	\end{align*}
	Thus, since by equations (\ref{eq:bound_on_grad_lipschitz}) and (\ref{eq:tanh_is_contracting}),
	\begin{align}
	\e\onenorm{\tanh\left(\grad f\left(X\right)\right)-\tanh\left(\grad f\left(Y\right)\right)} & \leq\e\onenorm{\grad f\left(X\right)-\grad f\left(Y\right)}\nonumber \\
	& \leq L_{2}\e\onenorm{X-Y}\nonumber \\
	& \leq32L_{2}\frac{n^{2/3}D^{1/3}}{\eps^{1/3}}.\label{eq:component_2}
	\end{align}
	Combining equations (\ref{eq:tanhx_is_close_to_etanhx}), (\ref{eq:component_1})
	and (\ref{eq:component_2}) together with the triangle inequality
	finally gives
	\begin{align*}
	\e\onenorm{\tanh\left(\grad f\left(Y\right)\right)-\e Y} & \leq32\left(1+L_{2}\right)\frac{n^{2/3}D^{1/3}}{\eps^{1/3}}+\eps n\\
	& \leq64L_{2}\frac{n^{2/3}D^{1/3}}{\eps^{1/3}}+\eps n
	\end{align*}
	as needed.
\end{proof}
\begin{lem}
	\label{lem:variance_of_tanh}Let $Z$ be an almost-surely bounded
	random variable, $\abs Z\leq L$ with $L\geq1$. Then 
	\begin{align*}
	\abs{\tanh\left(\e Z\right)-\e\tanh\left(Z\right)} & \leq20L\cdot\e\abs{\tanh\left(Z\right)-\e\tanh\left(Z\right)}.
	\end{align*}
\end{lem}

The proof is postponed to the appendix.
\begin{claim}
	\label{claim:exchange_expectation_of_product_measure}Let $\xi$ be
	a product measure on $\mathcal C_{n}$, let $Y\sim\xi$, and let $f:\mathcal C_{n}\to\bR$
	be a function on the hypercube. Then 
	\begin{equation}
	\e f\left(Y\right)=f\left(\e Y\right)\label{eq:expectation_switcharoo}
	\end{equation}
	and 
	\begin{equation}
	\e\grad f\left(Y\right)=\grad f\left(\e Y\right).\label{eq:expectation_switcharoo_gradient}
	\end{equation}
\end{claim}

\begin{proof}
	The extension of $f$ to $\overline{\mathcal C_{n}}$ is defined by the
	Fourier decomposition
	\[
	f\left(y\right)=\sum_{S\subseteq\left[n\right]}\hat{f}\left(S\right)\prod_{i\in S}y_{i}.
	\]
	Thus, since $\xi$ is a product measure, 
	\[
	\e f\left(Y\right)=\e\sum_{S\subseteq\left[n\right]}\hat{f}\left(S\right)\prod_{i\in S}Y_{i}=\sum_{S\subseteq\left[n\right]}\hat{f}\left(S\right)\prod_{i\in S}\e Y_{i}=f\left(\e Y\right).
	\]
	Equation \ref{eq:expectation_switcharoo_gradient} is then obtained
	by applying equation \ref{eq:expectation_switcharoo} to every component
	of $\grad f$.
\end{proof}
\begin{proof}[Proof Proposition \ref{prop:main_proposition}]
	By the triangle inequality, 
	\begin{align}
	\onenorm{\tanh\left(\grad f\left(\e Y\right)\right)-\e Y} & \leq\onenorm{\tanh\left(\grad f\left(\e Y\right)\right)-\e\tanh\left(\grad f\left(Y\right)\right)}+\onenorm{\e\tanh\left(\grad f\left(Y\right)\right)-\e Y}\nonumber \\
	\left(\text{by Claim \ref{claim:exchange_expectation_of_product_measure} }\right) & =\onenorm{\tanh\left(\e\grad f\left(Y\right)\right)-\e\tanh\left(\grad f\left(Y\right)\right)}+\onenorm{\e\tanh\left(\grad f\left(Y\right)\right)-\e Y}\nonumber \\
	\left(\text{by convexity}\right) & \leq\onenorm{\tanh\left(\e\grad f\left(Y\right)\right)-\e\tanh\left(\grad f\left(Y\right)\right)}+\e\onenorm{\tanh\left(\grad f\left(Y\right)\right)-\e Y}.\label{eq:the_final_breakdown}
	\end{align}
	Proposition \ref{prop:part_one} gives a bound on the second term
	in the right hand side.
	
	For the first term, note that by equation (\ref{eq:bound_on_lipschitz}),
	for every index $j\in\left[n\right]$, 
	\[
	\abs{\grad f\left(Y\right)_{j}}\leq L_{1}.
	\]
	We can therefore invoke Lemma \ref{lem:variance_of_tanh} on every
	index, giving that
	\begin{align}
	\onenorm{\tanh\left(\e\grad f\left(Y\right)\right)-\e\tanh\left(\grad f\left(Y\right)\right)} & =\sum_{j=1}^{n}\abs{\tanh\left(\e\grad f\left(Y\right)\right)_{j}-\e\tanh\left(\grad f\left(Y\right)\right)_{j}}\nonumber \\
	\left(\text{by Lemma \ref{lem:variance_of_tanh}}\right) & \leq20L_{1}\sum_{j=1}^{n}\e\abs{\tanh\left(\grad f\left(Y\right)\right)_{j}-\e\tanh\left(\grad f\left(Y\right)_{j}\right)}\nonumber \\
	& =20L_{1}\e\onenorm{\tanh\left(\grad f\left(Y\right)\right)-\e\tanh\left(\grad f\left(Y\right)\right)}.\label{eq:almost_there}
	\end{align}
	For this last term, we again use the triangle inequality and equation
	(\ref{eq:component_1}), giving
	
	\begin{align*}
	\e\onenorm{\tanh\left(\grad f\left(Y\right)\right)-\e\tanh\left(\grad f\left(Y\right)\right)} & \leq\e\onenorm{\tanh\left(\grad f\left(Y\right)\right)-\e\tanh\left(\grad f\left(X\right)\right)}+\\
	& \,\,\,\,\,\,\,+\e\onenorm{\e\tanh\left(\grad f\left(X\right)\right)-\e\tanh\left(\grad f\left(Y\right)\right)}\\
	& \leq\eps n+\e\onenorm{\tanh\left(\grad f\left(Y\right)\right)-\e Y}+\\
	& \,\,\,\,\,\,\,+\e\onenorm{\tanh\left(\grad f\left(X\right)\right)-\tanh\left(\grad f\left(Y\right)\right)}
	\end{align*}
	which, by Proposition \ref{prop:part_one} and equation (\ref{eq:component_2}),
	gives
	\begin{align*}
	\e\onenorm{\tanh\left(\grad f\left(Y\right)\right)-\e\tanh\left(\grad f\left(Y\right)\right)} & \leq\eps n+64L_{2}\frac{n^{2/3}D^{1/3}}{\eps^{1/3}}+\eps n\\
	& \,\,\,\,\,\,\,+32L_{2}\frac{n^{2/3}D^{1/3}}{\eps^{1/3}}\\
	& \leq96L_{2}\frac{n^{2/3}D^{1/3}}{\eps^{1/3}}+2\eps n.
	\end{align*}
	Putting this into equation (\ref{eq:almost_there}), we have
	\begin{align*}
	\onenorm{\tanh\left(\e\grad f\left(Y\right)\right)-\e\tanh\left(\grad f\left(Y\right)\right)} & \leq40L_{1}\left(48L_{2}\frac{n^{2/3}D^{1/3}}{\eps^{1/3}}+\eps n\right).
	\end{align*}
	And finally, plugging in the bounds into equation (\ref{eq:the_final_breakdown}),
	we get
	\begin{align*}
	\onenorm{\tanh\left(\grad f\left(\e Y\right)\right)-\e Y} & \leq40L_{1}\left(48L_{2}\frac{n^{2/3}D^{1/3}}{\eps^{1/3}}+\eps n\right)\\
	& \,\,\,\,\,+64L_{2}\frac{n^{2/3}D^{1/3}}{\eps^{1/3}}+\eps n\\
	& \leq41L_{1}\left(112L_{2}\frac{n^{2/3}D^{1/3}}{\eps^{1/3}}+\eps n\right).
	\end{align*}
\end{proof}

\section{Proof of composition theorem\label{sec:composition_theorem}}

We will use two lemmas concerning the relation between $f$ and $h\circ f$.
The first is a discrete chain rule which will be central to our calculations:
\begin{lem}[Chain rule for discrete gradient]
	\label{lem:chain_rule}Let $f:\mathcal C_{n}\to\bR$ with $\text{Lip}\left(f\right)=L$
	and let $h:\bR\to\bR$ with $\abs{h''\left(x\right)}<B$. Then
\end{lem}

\begin{enumerate}
	\item For ever $y\in\mathcal C_{n}$, 
	\begin{equation}
	\norm{\grad\left(h\circ f\right)\left(y\right)-h'\left(f\left(y\right)\right)\grad f\left(y\right)}_{1}\leq BL^{2}n\label{eq:chain_rule_one_norm}
	\end{equation}
	and
	\begin{equation}
	\norm{\grad\left(h\circ f\right)\left(y\right)-h'\left(f\left(y\right)\right)\grad f\left(y\right)}_{2}\leq BL^{2}\sqrt{n}.\label{eq:chain_rule_two_norm}
	\end{equation}
	\item For every $x\in\overline{\mathcal C_{n}}$, 
	\begin{equation}
	\norm{\grad\left(h\circ f\right)\left(x\right)-h'\left(f\left(x\right)\right)\grad f\left(x\right)}_{1}\leq2BL^{2}n^{3/2}.\label{eq:chain_rule_expectation}
	\end{equation}
\end{enumerate}
The second lemma concerns the parameters of the function $h\circ f$:
\begin{lem}[Composition parameters]
	\label{lem:composition_parameters}Let $h:\bR\to\bR$ be a twice differentiable
	function satisfying 
	\begin{align*}
	\abs{h'\left(x\right)} & \leq B_{1}\\
	\abs{h''\left(x\right)} & \leq B_{2}
	\end{align*}
	for all $x\in\bR$. Let $f:\mathcal C_{n}\to\bR$ be a function with parameters
	$D$, $L_{1}$, $L_{2}$ as described in Theorem \ref{thm:main_theorem}.
	Then
	\begin{align}
	\mathcal D\left(h\circ f\right) & \leq B_{1}D+B_{2}L_{1}^{2}n\label{eq:composition_parameter_1}\\
	\text{Lip}\left(h\circ f\right) & \leq B_{1}L_{1}\label{eq:composition_parameter_2}\\
	\max_{x\neq y\in\mathcal C_{n}}\frac{\onenorm{\grad\left(h\circ f\right)\left(x\right)-\grad\left(h\circ f\right)\left(y\right)}}{\onenorm{x-y}} & \leq B_{1}L_{2}+3B_{2}L_{1}^{2}n.\label{eq:composition_parameter_3}
	\end{align}
	
\end{lem}

The proofs of both lemmas are postponed to the appendix. 
\begin{proof}[Proof of Theorem \ref{thm:composition_theorem}]
	Denote by $\mathcal X_{h\circ f}$ the set 
	\[
	\mathcal{\mathcal X}_{h\circ f}=\left\{ X\in\left[-1,1\right]^{n}:\onenorm{X-\tanh\left(\grad\left(h\circ f\right)\left(X\right)\right)}\leq5000\tilde{L}_{1}\tilde{L}_{2}^{3/4}\tilde{D}^{1/4}n^{3/4}\right\} .
	\]
	Note that for every $X\in\mathcal X_{h\circ f}$,
	
	\begin{align*}
	\onenorm{X-\tanh\left(h'\left(X\right)\grad f\left(X\right)\right)} & \leq\onenorm{X-\tanh\left(\grad\left(h\circ f\right)\left(X\right)\right)}\\
	& +\,\,\,\,\,\,\,\,\,\,\onenorm{\tanh\left(\grad\left(h\circ f\right)\left(X\right)\right)-\tanh\left(h'\left(X\right)\grad f\left(X\right)\right)}\\
	\left(\text{by equation (\ref{eq:chain_rule_expectation})}\right) & \leq5000\tilde{L}_{1}\tilde{L}_{2}^{3/4}\tilde{D}^{1/4}n^{3/4}+2B_{2}L_{1}^{2}n^{3/2}
	\end{align*}
	and so $\tilde{\mathcal X}_{h\circ f}\subseteq\mathcal X_{h\circ f}$. Applying
	Theorem \ref{thm:main_theorem} for the function $h\circ f$ with
	the bounds given by Lemma \ref{lem:composition_parameters} gives
	the required results. 
\end{proof}
\begin{rem}
	\label{rem:tightness_of_extension_expectation_bound}The bound for
	compositions $h\circ f$ with domain $\overline{C_{n}}$, given in
	(\ref{eq:chain_rule_expectation}), is worse by a factor of $\sqrt{n}$
	than that of compositions with domain $\mathcal{C}_{n}$, given in
	(\ref{eq:chain_rule_one_norm}). This disparity is in fact tight.
	For example, consider the function 
	\[
	h\left(x\right)=\begin{cases}
	\frac{3}{4}x^{3}-\frac{1}{4}x^{5} & \abs x<1\\
	\frac{1}{2}x^{2} & x\geq1\\
	-\frac{1}{2}x^{2} & x\leq-1
	\end{cases}
	\]
	applied to the ``counting'' function 
	\[
	f\left(x\right)=\sum_{i=1}^{n}x_{i}.
	\]
	The function $h$ has a bounded second derivative and satisfies $h'\left(0\right)=0$.
	For $x=0$, we have $f\left(x\right)=0$ and so $h'\left(f\left(x\right)\right)\grad f\left(x\right)=0$
	as well. However, a calculation shows that $\norm{\grad\left(h\circ f\right)\left(x\right)}_{1}\sim n^{3/2}$,
	and so $\norm{\grad\left(h\circ f\right)\left(x\right)-h'\left(f\left(x\right)\right)\grad f\left(x\right)}_{1}\sim n^{3/2}$
	as well. 
\end{rem}

\section{Example applications\label{sec:applications}}

\subsection{The Ising model\label{subsec:ising_proof}}
\begin{proof}[Proof of Corollary \ref{cor:general_ising}]
	A short calculation shows that $\grad f\left(x\right)=Ax+\mu$. The
	corollary will follow immediately from Theorem \ref{thm:main_theorem}
	once we have obtained the parameters $D,$ $L_{1}$ and $L_{2}$ for
	$f$. The calculations for $\mathcal D\left(f\right)$ and $\text{Lip}$$\left(f\right)$
	are also found in \cite[Section 1.3]{2016_eldan_gaussian_width} but
	we repeat them here for completeness.
	
	Denote $\mu_{\max}=\max_{i\in\left[n\right]}\abs{\mu_{i}}$. We then
	have
	\begin{enumerate}
		\item The Gaussian-width is bounded by: 
		\begin{align*}
		\mathcal D\left(f\right) & =\e\sup_{x\in\mathcal C_{n}}\left\langle Ax+\mu,\Gamma\right\rangle \\
		& \leq\e\sup_{x\in\mathcal C_{n}}\left\langle Ax,\Gamma\right\rangle +\e\abs{\left\langle \mu,\Gamma\right\rangle }\\
		& \leq\sqrt{n}\e\sup_{x\in B\left(0,1\right)}\left\langle Ax,\Gamma\right\rangle +\norm{\mu}_{2}\\
		& =\sqrt{n}\e\norm{A\Gamma}_{2}+\twonorm{\mu}\\
		& \leq\sqrt{n\e\norm{A\Gamma}_{2}^{2}}+\twonorm{\mu}\\
		& \leq\sqrt{n\tr A^{2}}+\sqrt{n}\mu_{\max}.
		\end{align*}
		\item The Lipschitz constant is bounded by
		\begin{align*}
		\text{Lip}\left(f\right) & \leq\mu_{\max}+\max_{i\in\left[n\right],x\in\mathcal C_{n}}\left\langle Ax,e_{i}\right\rangle \\
		& \leq\mu_{\max}+\max_{i\in\left[n\right]}\sum_{j\in\left[n\right]}\abs{A_{ij}}.
		\end{align*}
		\item Regarding the Lipschitz constant of the gradient, note that $\onenorm{\grad f\left(x\right)-\grad f\left(y\right)}=\onenorm{A\left(x-y\right)}$.
		Suppose that $x$ and $y$ differ only in the $i$-th coordinate.
		Then $A\left(\abs{x-y}\right)$ is just 2 times the $i$-th column
		of $A$. By the triangle inequality, we then have
		\[
		\frac{\onenorm{\grad f\left(x\right)-\grad f\left(y\right)}}{\onenorm{x-y}}\leq\max_{i\in\left[n\right]}\sum_{j\in\left[n\right]}\abs{A_{ij}}.
		\]
	\end{enumerate}
\end{proof}
\begin{proof}[Proof of Corollary \ref{cor:curie_wiess_ising}]
	The interactions described in Corollary \ref{cor:curie_wiess_ising}
	can be represented by an interaction matrix $A=\frac{\beta\one}{n}$,
	where $\one$ is the $n\times n$ matrix whose off-diagonal entries
	are $1$ and whose diagonal is $0$, and $\beta$ is interpreted as
	the inverse temperature. Note that for every $x,y\in\mathcal C_{n}$,
	\begin{equation}
	\onenorm{\grad f\left(x\right)-\grad f\left(y\right)}=\onenorm{A\left(x-y\right)}\leq\beta\onenorm{x-y},\label{eq:curry_weiss_contracting}
	\end{equation}
	so that $L_{2}\leq1+\beta$. A simple calculation also shows that
	$D\leq\beta\sqrt{n}$ and $L_{1}\leq1+\beta$. Denoting 
	\[
	\mathcal X=\left\{ X\in\overline{\mathcal C_{n}}:\onenorm{X-\tanh\left(\frac{\beta\one}{n}X\right)}\leq5000\left(1+\beta\right)^{2}n^{7/8}\right\} ,
	\]
	by Corollary \ref{cor:general_ising} we have that $X_{n}^{f}$ is
	a $\left(\rho,3n^{-1/8},3n^{-1/8}\right)$-tilt-mixture with $\rho\left(\mathcal X\right)\geq1-3n^{-1/8}$.
	Denote by $\boldsymbol{J}=\one+\textrm{Id}$ the $n\times n$ matrix
	whose every entry is $1$. Then every $X\in\mathcal X$ also satisfies
	\begin{align*}
	\onenorm{X-\tanh\left(\frac{\beta\boldsymbol{J}}{n}X\right)} & =\onenorm{X-\tanh\left(\frac{\beta\boldsymbol{J}}{n}X\right)-\tanh\left(\frac{\beta\one}{n}X\right)+\tanh\left(\frac{\beta\one}{n}X\right)}\\
	& \leq5000\left(1+\beta\right)^{2}n^{7/8}+\onenorm{\tanh\left(\frac{\beta\one}{n}X\right)-\tanh\left(\frac{\beta\left(\one+\textrm{Id}\right)}{n}X\right)}\\
	& \leq5000\left(1+\beta\right)^{2}\left(1+\beta\right)n^{7/8}+\onenorm{\frac{\beta\textrm{Id}}{n}X}\\
	& \leq5001\left(1+\beta\right)^{2}n^{7/8}.
	\end{align*}
	Thus $\mathcal X\subseteq\mathcal X_{f}$ and the first part of Corollary
	\ref{cor:curie_wiess_ising} is proved. The fixed point equation $X=\tanh\left(\frac{\beta\boldsymbol{J}}{n}X\right)$
	is easier to work with, since all of its exact solutions are constant:
	Indeed, every entry $X_{i}$ of a solution satisfies $X_{i}=\tanh\left(\sum_{j=1}^{n}\frac{\beta}{n}X_{j}\right)$;
	every solution $X$ is then of the form $X=\left(x,x,\ldots,x\right)$,
	and the exact fixed point vector equation reduces to the scalar equation
	\[
	x=\tanh\left(\beta x\right).
	\]
	The value $x_{0}=0$ is always a solution, corresponding to the case
	where the typical configuration is completely disordered. 
	
	For $\beta\leq1$, this is also the only solution. In this case, for
	every $X\in\mathcal X_{f}$,
	\begin{align*}
	\onenorm X & =\onenorm{X-\tanh\left(\frac{\beta\boldsymbol{J}}{n}X\right)+\tanh\left(\frac{\beta\boldsymbol{J}}{n}X\right)}\\
	& \leq5001\left(1+\beta\right)^{2}n^{7/8}+\onenorm{\tanh\left(\frac{\beta\boldsymbol{J}}{n}X\right)}\\
	& \leq5001\left(1+\beta\right)^{2}n^{7/8}+\beta\onenorm X.
	\end{align*}
	Rearranging, we get that every $X\in\mathcal X_{f}$ is close to $0$:
	\[
	\onenorm X\leq5001\frac{\left(1+\beta\right)^{2}}{1-\beta}n^{7/8}.
	\]
	This represents the fact that for high temperatures, the system is
	always disordered.
	
	For $\beta>1$, there are two other solutions, $x_{1}=-x_{2}$. These
	satisfy $\abs{x_{1}},\abs{x_{2}}\to1$ as $\beta\goinf$, and correspond
	to the symmetry-broken phase where all spins tend to point in the
	same direction. Showing that every $X\in\mathcal X_{f}$ is close to either
	$\left(x_{1},x_{1},\ldots,x_{1}\right)$ or $\left(x_{2},x_{2},\ldots,x_{2}\right)$
	can then be done by a standard counting argument, which we choose
	to omit.
\end{proof}
Adding a constant magnetic field $\mu=\left(\mu_{0},\mu_{0}\ldots,\mu_{0}\right)$
forces a non-zero constant solution for every $\beta>0$, while shifting
the values of $x_{1}$ and $x_{2}$.

\subsection{Large deviations\label{subsec:large_deviations_proof}}

In order to prove Theorem \ref{thm:large_deviations}, we follow the
approach mentioned in Section \ref{subsec:large_deviations_intro_section},
and try to approximate function $\tilde{f}$ in equation (\ref{eq:the_sharpest_cutoff})
by a well-behaved Hamiltonian $g$. 

Let $t\in\bR$ and $\delta>0$. Let $h:\bR\to\bR$ and $\psi:\bR\to\bR$
be defined as 

\[
h\left(x\right)=\begin{cases}
2x+1 & x\leq-1\\
-x^{2} & -1\leq x\leq0\\
0 & x\geq0.
\end{cases}
\]
and
\[
\psi\left(x\right)=n\cdot h\left(\left(\frac{x}{n}-t\right)/\delta\right).
\]
Note that $\abs{h'\left(x\right)}\leq2$ and $\abs{h''\left(x\right)}\leq2$
for all $x\in\bR$. Thus 
\begin{align*}
\abs{\psi'\left(x\right)} & =\abs{n\cdot h'\left(\left(\frac{x}{n}-t\right)/\delta\right)\cdot\frac{1}{n\delta}}\\
& \leq\frac{2}{\delta}
\end{align*}
and 
\begin{align*}
\abs{\psi''\left(x\right)} & =\abs{n\cdot h''\left(\left(\frac{x}{n}-t\right)/\delta\right)\cdot\frac{1}{n^{2}\delta^{2}}}\\
& \leq\frac{2}{n\delta^{2}}.
\end{align*}
Let $g:\mathcal C_{n}\to\bR$ be defined as 
\[
g\left(y\right)=\psi\left(f\left(y\right)\right).
\]
Denote by $\nu$ the measure defined by $X_{n}^{g}$. The function
$g$ is an approximation for $\tilde{f}$, in the sense that almost
of all the mass of $\nu$ is supported on vectors on which $f$ attains
a large value. 
\begin{prop}
	\label{prop:small_tail}Let $\delta'=\frac{\log4+1}{2}\delta$ and
	define 
	\[
	\mathcal B=\left\{ y\in\mathcal C_{n}:\,f\left(y\right)\leq\left(t-\delta'\right)n\right\} .
	\]
	If there exists a $z\in\mathcal C_{n}$ such that $f\left(z\right)\geq tn$,
	then 
	\[
	\nu\left(\mathcal B\right)\leq2^{-n}.
	\]
\end{prop}

\begin{proof}
	Let $y\in\mathcal B$. By definition of $g$, 
	\begin{align*}
	g\left(y\right) & =n\cdot h\left(\left(\frac{f\left(y\right)}{n}-t\right)/\delta\right)\\
	\left(\text{\ensuremath{h} is increasing}\right) & \leq n\cdot h\left(\left(\frac{\left(t-\delta'\right)n}{n}-t\right)/\delta\right)\\
	& =n\cdot h\left(-\delta'/\delta\right)\\
	\left(\text{since \ensuremath{\delta'>\delta}}\right) & =n\left(1-2\frac{\delta'}{\delta}\right)\\
	& =-\left(\log4\right)n.
	\end{align*}
	Let $z\in\mathcal C_{n}$ be such that $f\left(z\right)\geq tn$. Then
	under $\nu$ the probability for obtaining $z$ is proportional to
	$e^{g\left(z\right)}=e^{0}=1$. On the other hand, for every $y\in\mathcal B$,
	the probability for obtaining $y$ is proportional to a value smaller
	than $e^{-\log4\cdot n}=4^{-n}=2^{-2n}$. Since there are no more
	than $2^{n}$ possible vectors in $\mathcal C_{n}$, we thus obtain
	\[
	\nu\left(\mathcal B\right)\leq\frac{\nu\left(\mathcal B\right)}{\nu\left(z\right)}\leq2^{n}2^{-2n}=2^{-n}.
	\]
\end{proof}
Proposition \ref{prop:small_tail} allows us to approximate $\nu$
with a distribution that does not give any mass at all to vectors
$y\in\mathcal C_{n}$ with $f\left(y\right)<\left(t-\delta'\right)n$.
Define the function $\varphi:\mathcal C_{n}\to\bR$ by

\[
\varphi\left(y\right)=\begin{cases}
0 & f\left(y\right)<\left(t-\delta'\right)n\\
e^{g\left(y\right)} & \left(t-\delta'\right)n\leq f\left(y\right)<tn\\
1 & f\left(y\right)\geq tn,
\end{cases}
\]
and observe that $\varphi\left(y\right)$ agrees with $e^{g\left(y\right)}$
for all $y$ such that $f\left(y\right)\geq\left(t-\delta'\right)n$.
Denote by $\sigma$ the measure defined by $d\sigma=\frac{\varphi d\mu}{\int_{\mathcal C_{n}}\varphi d\mu}$
and by $X_{\varphi}$ a random variable whose law is $\sigma$.
\begin{prop}
	\label{prop:phi_and_g_are_close}Assume that there exists a $z\in\mathcal C_{n}$
	such that $f\left(z\right)\geq tn$. Then there exists a coupling
	between $X_{n}^{g}$ and $X_{\varphi}$ such that 
	\[
	\e\onenorm{X_{n}^{g}-X_{\varphi}}\leq2n\cdot2^{-n}.
	\]
\end{prop}

We postpone the proof to the appendix.
\begin{proof}[Proof of Theorem \ref{thm:large_deviations}]
	Applying Theorem \ref{thm:composition_theorem} to $g$, there exists
	a $\rho$-mixture and a coupling between $X\left(\rho\right)$ and
	$X_{n}^{g}$ such that 
	\begin{equation}
	\rho\left(\mathcal X_{g}\right)\geq1-\frac{80\tilde{D}^{1/4}}{n^{1/4}}\label{eq:middleway_probability}
	\end{equation}
	and
	\[
	\e\onenorm{X\left(\rho\right)-X_{n}^{g}}\leq80n^{3/4}\tilde{D}^{1/4}.
	\]
	Therefore by Proposition \ref{prop:phi_and_g_are_close} there exists
	a coupling between $X\left(\rho\right)$ and $X_{\varphi}$ such that
	\[
	\e\onenorm{X\left(\rho\right)-X_{\varphi}}\leq80n^{3/4}\tilde{D}^{1/4}+2n\cdot2^{-n}.
	\]
	This shows that $X_{\varphi}$ is a $\left(\rho,80\frac{\tilde{D}^{1/4}}{n^{1/4}}+8\cdot2^{-n}\right)$-mixture.
	To obtain equation (\ref{eq:large_deviation_probability_statement}),
	denote $\mathcal Y_{g}=\left\{ X\in\mathcal X_{g}:\,f\left(X\right)<\left(t-3\delta'\right)n\right\} $,
	and let $X\in\mathcal Y_{g}$. Denote by $\xi_{X}$ the product measure
	on $\mathcal C_{n}$ such that if $Y_{X}\sim\xi_{X}$ then $\e Y_{X}=X$.
	We then have
	\begin{align}
	\pr{f\left(Y_{X}\right)\geq\left(t-2\delta'\right)n} & \leq\pr{f\left(Y_{X}\right)\geq f\left(\e Y_{X}\right)+\delta'n}\nonumber \\
	& \leq\pr{\abs{f\left(Y_{X}\right)-f\left(\e Y_{X}\right)}\geq\delta'n}\nonumber \\
	\left(\text{by Markov's inequality}\right) & \leq\frac{\e\abs{f\left(Y_{X}\right)-f\left(\e Y_{X}\right)}}{\delta'n}\nonumber \\
	\left(\text{by Proposition \ref{prop:variance_of_f} }\right) & \leq\frac{L_{1}}{\delta'\sqrt{n}}.\label{eq:above_2_threshold}
	\end{align}
	Denote by $\mathcal A_{X}$ the event 
	\[
	\mathcal A_{X}=\left\{ f\left(Y_{X}\right)<\left(t-2\delta'\right)n\right\} \intersect\left\{ f\left(X_{n}^{g}\right)>\left(t-\delta'\right)n\right\} .
	\]
	Equation (\ref{eq:above_2_threshold}) and Proposition \ref{prop:small_tail}
	together imply that
	\[
	\pr{\mathcal A_{X}}\geq1-\frac{L_{1}}{\delta'\sqrt{n}}-2^{-n}.
	\]
	Under $\mathcal A_{X}$ we have that
	
	\[
	\delta'n\leq f\left(X_{n}^{g}\right)-f\left(Y_{X}\right)\leq L_{1}\onenorm{Y_{X}-X_{n}^{g}},
	\]
	yielding
	\begin{equation}
	\onenorm{Y_{X}-X_{n}^{g}}\geq\frac{\delta'n}{L_{1}}.\label{eq:not_too_sharp}
	\end{equation}
	Since $\e\onenorm{X\left(\rho\right)-X_{n}^{g}}$ is small, this inequality
	sets a constraint on the measure of $\mathcal Y_{g}$. Letting $Z$ be
	a random variables with law $\rho$, coupled with $X\left(\rho\right)$
	so that $X\left(\rho\right)\mid Z\sim Y_{Z}$, one has
	
	\begin{align*}
	\e\onenorm{X\left(\rho\right)-X_{n}^{g}} & =\int_{\overline{\mathcal C_{n}}}\e\left[\onenorm{Y_{Z}-X_{n}^{g}}\mid Z\right]d\rho\left(Z\right)\\
	& \geq\int_{\mathcal Y_{g}}\e\left[\onenorm{Y_{Z}-X_{n}^{g}}\mid Z\right]d\rho\left(Z\right)\\
	& \geq\int_{\mathcal Y_{g}}\e\left[\onenorm{Y_{Z}-X_{n}^{g}}\mid Z\land\mathcal A_{Z}\right]\pr{\mathcal A_{Z}}d\rho\left(Z\right)\\
	\left(\text{by equation (\ref{eq:not_too_sharp})}\right) & \geq\left(1-\frac{L_{1}}{\delta'\sqrt{n}}-2^{-n}\right)\int_{\mathcal Y_{g}}\frac{\delta'n}{L_{1}}d\rho\left(Z\right)\\
	& =\left(1-\frac{L_{1}}{\delta'\sqrt{n}}-2^{-n}\right)\rho\left(\mathcal Y_{g}\right)\frac{\delta'n}{2L_{1}}.
	\end{align*}
	We thus obtain
	\[
	\rho\left(\mathcal Y_{g}\right)\leq\frac{80L_{1}\tilde{D}^{1/4}}{n^{1/4}\delta'}\left(1-\frac{L_{1}}{\delta'\sqrt{n}}-2^{-n}\right)^{-1}.
	\]
	Together with equation (\ref{eq:middleway_probability}), this gives
	\[
	\rho\left(\mathcal X_{g}\backslash\mathcal Y_{g}\right)\geq1-\frac{161L_{1}\tilde{D}^{1/4}}{n^{1/4}\delta'}\left(1-\frac{L_{1}}{\delta'\sqrt{n}}-2^{-n}\right)^{-1}
	\]
	as needed.
\end{proof}
\begin{rem}
	A particular type of Hamiltonian that has been of considerable interest
	in the field of large deviations that of subgraph-counting functions.
	It was recently shown in \cite{eldan_gross_exponential_random_graphs}
	that for these types of Hamiltonians, $\grad f\left(X\right)$ is
	close to a stochastic block matrix. Since $h'\left(\left(\frac{f\left(X\right)}{n}-t\right)/\delta\right)$
	is a scalar, this implies that every $X\in\mathcal X_{g}$ is also close
	to a stochastic block matrix. 
\end{rem}

\begin{rem}
	Theorem \ref{thm:large_deviations} corresponds to the unconditioned
	distribution $\mu_{p}$ with $p=1/2$. To deal with the case $p\neq1/2$,
	define $g\left(y\right)$ as 
	\[
	g\left(y\right)=\psi\left(f\left(y\right)\right)+\prod\log\left(\frac{1}{2}\left(1-y_{i}+2py_{i}\right)\right).
	\]
	Analogues of Propositions \ref{prop:small_tail} and \ref{prop:phi_and_g_are_close}
	can then be proved following the same line.
\end{rem}

\section{Acknowledgments}
The first author is grateful to Sourav Chatterjee for inspiring him
to work on this topic and for an enlightening discussion. We thank
Amir Dembo and Sumit Mukherjee for insightful discussions, and Yufei Zhao for 
his motivating comments on sparse bounds. Finally we thank the anonymous 
referee for comments improving the presentation	of this work.

\addcontentsline{toc}{section}{\protect\numberline{}References}

\bibliographystyle{plain}
\bibliography{large_bibliography}

\appendix

\section{Appendix}

 \begin{proof}[Proof of Lemma \ref{lem:variance_of_tanh}]
 	
 	Denote $Y=\tanh Z$, and denote the bound of $Y$ by $\alpha=\tanh L\geq\tanh\left(1\right)$.
 	Under this notation, we wish to show that 
 	
 	\begin{equation}
 	\abs{\tanh\left(\e\tanh^{-1}Y\right)-\e Y}\leq20\tanh^{-1}\left(\alpha\right)\cdot\e\abs{Y-\e Y}.\label{eq:what_we_want_to_show_under_Y}
 	\end{equation}
 	We will prove this inequality by considering it as a variational problem
 	on the distribution $\mu$ of $Y$. Specifically, we will show that
 	for every $a\in\left[-\alpha,\alpha\right]$, every distribution $\mu$
 	of $Y$ satisfies 
 	\begin{equation}
 	\abs{\tanh\left(\e\tanh^{-1}Y\right)-a}\leq20\tanh^{-1}\left(\alpha\right)\cdot\e\abs{Y-a}.\label{eq:how_to_show_what_we_want}
 	\end{equation}
 	Setting $a=\e Y$ then gives the desired result. 
 	
 	Suppose that $\e\abs{Y-a}$ is fixed. Then the left hand side of (\ref{eq:how_to_show_what_we_want})
 	is maximized by the $Y$ that gives $\tanh\left(\e\tanh^{-1}Y\right)$
 	an extremal value, conditioned on $b:=\e\abs{Y-a}$ being constant.
 	Since $\tanh$ is monotone, this is equivalent to finding the extremal
 	value of the integral 
 	\begin{equation}
 	\int\tanh^{-1}\left(x\right)d\mu\left(x\right)\label{eq:linear_thing_to_maximize}
 	\end{equation}
 	while maintaining the constraint 
 	\begin{equation}
 	b=\e\abs{Y-a}.\label{eq:linear_constraint_on_Y}
 	\end{equation}
 	The constraint (\ref{eq:linear_constraint_on_Y}) is of the form $\int f\left(x\right)d\mu=b$,
 	where $f\left(x\right)=\abs{x-a}$. By Theorems 2.1 and 3.2 and Proposition
 	3.1 in \cite{gerhard_1988}, the extremal distributions which solve
 	a system of $n$ constraints of the form $\int f_{i}\left(x\right)d\mu=c_{i}$
 	are linear combinations of no more than $n+1$ singletons, i.e delta
 	distributions. We can therefore write the extremal $\text{\ensuremath{\mu}}$
 	as 
 	\begin{equation}
 	\mu=p\delta\left(x\right)+\left(1-p\right)\delta\left(y\right)\label{eq:decomposition_of_mu}
 	\end{equation}
 	for some two real numbers $-\alpha\leq x,y\leq\alpha$ and $p\in\left[0,1\right]$.
 	Now, using the triangle inequality, we have that 
 	\[
 	\abs{\tanh\left(\e\tanh^{-1}Y\right)-a}\leq\abs{\tanh\left(\e\tanh^{-1}Y\right)-\e Y}+\e\abs{Y-a},
 	\]
 	so it is in fact enough to show that 
 	\[
 	\abs{\tanh\left(\e\tanh^{-1}Y\right)-\e Y}\leq19\tanh^{-1}\left(\alpha\right)\cdot\e\abs{Y-a},
 	\]
 	and since $\e\abs{Y-\e Y}\leq2\e\abs{Y-a}$ for every $a$, it actually
 	suffices to show that
 	\begin{equation}
 	\abs{\tanh\left(\e\tanh^{-1}Y\right)-\e Y}\leq9\tanh^{-1}\left(\alpha\right)\cdot\e\abs{Y-\e Y}.\label{eq:using_expectation_instead_of_a}
 	\end{equation}
 	Plugging the decomposition (\ref{eq:decomposition_of_mu}) into (\ref{eq:using_expectation_instead_of_a}),
 	we need to prove that for every such $x$ and $y$,
 	\[
 	\frac{\abs{\tanh\left(p\tanh^{-1}\left(x\right)+\left(1-p\right)\tanh^{-1}\left(y\right)\right)-\left(px+\left(1-p\right)y\right)}}{2p\left(1-p\right)\abs{x-y}\tanh^{-1}\left(\alpha\right)}\leq9.
 	\]
 	Assume without loss of generality that $x>0$ and $x>\abs y$. We
 	will now show that inequality is correct for $0<p\leq\frac{1}{2}$.
 	We omit the similar proof for $\frac{1}{2}\leq p<1$. For these values
 	of $p$, it suffices to show that 
 	\begin{equation}
 	\frac{\abs{\tanh\left(p\tanh^{-1}\left(x\right)+\left(1-p\right)\tanh^{-1}\left(y\right)\right)-\left(px+\left(1-p\right)y\right)}}{p\tanh^{-1}\left(\alpha\right)\left(x-y\right)}\leq9\cdot\label{eq:p_smaller_than_half}
 	\end{equation}
 	For every fixed value of $y$, we treat the expression on the left
 	hand side as a function of $p$ for $p\in\left(0,1\right)$. This
 	expression may attain its supremum either at $p\to0^{+}$, $p=\frac{1}{2}$,
 	or at values of $p$ such that the derivative of the left hand side
 	with respect to $p$ is 0. We'll now consider each of these three
 	cases.
 	
 	\subsubsection*{Taking the derivative}
 	
 	Comparing the derivative to $0$, one obtains the relation
 	
 	\begin{align*}
 	\tanh\left(p\tanh^{-1}\left(x\right)+\left(1-p\right)\tanh^{-1}\left(y\right)\right)-\left(px+\left(1-p\right)y\right) & =\\
 	\left(\frac{\tanh^{-1}\left(x\right)-\tanh^{-1}\left(y\right)}{\cosh^{2}\left(p\tanh^{-1}\left(x\right)+\left(1-p\right)\tanh^{-1}\left(y\right)\right)}-\left(x-y\right)\right)p.
 	\end{align*}
 	Plugging this back into (\ref{eq:p_smaller_than_half}) and using
 	the triangle inequality, it is enough to show that
 	\begin{equation}
 	\frac{\frac{\tanh^{-1}\left(x\right)-\tanh^{-1}\left(y\right)}{\cosh^{2}\left(p\tanh^{-1}\left(x\right)+\left(1-p\right)\tanh^{-1}\left(y\right)\right)}+\left(x-y\right)}{\tanh^{-1}\left(\alpha\right)\left(x-y\right)}\leq9.\label{eq:taking_derivative_part_way}
 	\end{equation}
 	Since $\tanh^{-1}\left(\alpha\right)\geq1$, the expression $\frac{\left(x-y\right)}{\tanh^{-1}\left(\alpha\right)\left(x-y\right)}$
 	is bounded by $1$, so it remains to show that 
 	\begin{equation}
 	\frac{\tanh^{-1}\left(x\right)-\tanh^{-1}\left(y\right)}{\tanh^{-1}\left(\alpha\right)\cosh^{2}\left(p\tanh^{-1}\left(x\right)+\left(1-p\right)\tanh^{-1}\left(y\right)\right)\left(x-y\right)}\leq8.\label{eq:tanhz_involving_cosh2}
 	\end{equation}
 	If $y<0$ and $x\geq\frac{1}{2}$, then $x-y>1/2$ and we trivially
 	have
 	\[
 	\frac{\tanh^{-1}\left(x\right)-\tanh^{-1}\left(y\right)}{\tanh^{-1}\left(\alpha\right)}\frac{1}{\cosh^{2}\left(p\tanh^{-1}\left(x\right)+\left(1-p\right)\tanh^{-1}\left(y\right)\right)\left(x-y\right)}\leq\frac{2}{\frac{1}{2}}=4.
 	\]
 	If $y<0$ and $x<\frac{1}{2}$ then 
 	\[
 	\tanh^{-1}\left(x\right)-\tanh^{-1}\left(y\right)\leq\frac{1}{1-x^{2}}\left(x-y\right)\leq\frac{4}{3}\left(x-y\right)
 	\]
 	and so 
 	\[
 	\frac{\tanh^{-1}\left(x\right)-\tanh^{-1}\left(y\right)}{\tanh^{-1}\left(\alpha\right)\cosh^{2}\left(p\tanh^{-1}\left(x\right)+\left(1-p\right)\tanh^{-1}\left(y\right)\right)\left(x-y\right)}\leq\frac{2}{\tanh^{-1}\left(\alpha\right)}<8.
 	\]
 	For $y\geq0$, the maximal w.r.t $p$ value of the left hand side
 	of (\ref{eq:tanhz_involving_cosh2}) is attained when the argument
 	of $\cosh^{2}$ is minimal, i.e at $p=0$. Using the fact that $\cosh\left(\tanh^{-1}\left(y\right)\right)=1/\sqrt{1-y^{2}}>1/\sqrt{2\left(1-y\right)}$
 	and that $\tanh^{-1}\left(x\right)=\frac{1}{2}\log\frac{1+x}{1-x}$,
 	it suffices to show that 
 	
 	\begin{equation}
 	\frac{\left(\log\frac{1+x}{1-x}\frac{1-y}{1+y}\right)\left(1-y\right)}{\tanh^{-1}\left(\alpha\right)\left(x-y\right)}\leq8.\label{eq:temp_tanh_0}
 	\end{equation}
 	We consider two cases. Suppose that $\frac{1-y}{1-x}\geq2$. For any
 	$z\geq2$, it holds that $\log2z\leq2\log z$, and since $x,y<1$,
 	it is enough to show that
 	\begin{equation}
 	2\frac{\left(\log\frac{1-y}{1-x}\right)\left(1-y\right)}{\tanh^{-1}\left(\alpha\right)\left(x-y\right)}\leq8.\label{eq:temp_tanh_1}
 	\end{equation}
 	Denote $1-x=e^{-a}$ and $1-y=e^{-b}$, with $a>b>0$; under this
 	notation, the left hand side becomes $2\frac{\left(a-b\right)}{\tanh^{-1}\left(\alpha\right)\left(1-e^{-\left(a-b\right)}\right)}$.
 	Note that $\tanh^{-1}\left(\alpha\right)=\frac{1}{2}\log\frac{1+\alpha}{1-\alpha}\geq\frac{1}{2}\log\frac{1}{1-x}=\frac{1}{2}a$.
 	If $e^{-\left(a-b\right)}<\frac{1}{2}$, then $2\frac{\left(a-b\right)}{\tanh^{-1}\left(\alpha\right)\left(1-e^{-\left(a-b\right)}\right)}\leq2\frac{\left(a-b\right)}{\frac{1}{2}a\left(\frac{1}{2}\right)}\leq8$.
 	Otherwise, if $e^{-\left(a-b\right)}\geq\frac{1}{2}$, then $a-b<\frac{3}{4}$.
 	By Taylor's theorem, the $1-e^{-\left(a-b\right)}$ in the denominator
 	can be bounded from below by $\frac{1}{2}\left(a-b\right)$, bounding
 	the expression by $\frac{8}{\tanh^{-1}\left(\alpha\right)}\leq8$. 
 	
 	Now suppose that $\frac{1-y}{1-x}<2$. Since $\log z\leq z-1$ for
 	all $z$, we may then write the left hand side of (\ref{eq:temp_tanh_0})
 	as
 	\begin{align*}
 	\frac{\left(\log\frac{1+x}{1+y}+\log\frac{1-y}{1-x}\right)\left(1-y\right)}{\tanh^{-1}\left(\alpha\right)\left(x-y\right)} & \leq\frac{1}{\tanh^{-1}\left(\alpha\right)}\left(\left(\frac{1+x}{1+y}-1\right)+\left(\frac{1-y}{1-x}-1\right)\right)\frac{1-y}{x-y}\\
 	& \leq\frac{1}{\tanh^{-1}\left(\alpha\right)}\left(\frac{1-y}{1+y}+\frac{1-y}{1-x}\right)\\
 	& \leq\frac{3}{\tanh^{-1}\left(\alpha\right)}<8.
 	\end{align*}
 	
 	\subsubsection*{The case $p=0$}
 	
 	Using L'Hôpital's rule, the value of the left hand side of (\ref{eq:p_smaller_than_half})
 	attained as $p\to0^{+}$ is 
 	\[
 	\frac{\abs{\frac{\tanh^{-1}\left(x\right)-\tanh^{-1}\left(y\right)}{\cosh^{2}\left(\tanh^{-1}\left(y\right)\right)}-\left(x-y\right)}}{\tanh^{-1}\left(\alpha\right)\left(x-y\right)}.
 	\]
 	For $y\geq0$, this is the same expression obtained by setting $p=0$
 	in (\ref{eq:taking_derivative_part_way}). The case $y<0$ is handled
 	similarly as above.
 	
 	\subsubsection*{The case $p=1/2$}
 	
 	In this case we must show that 
 	\[
 	\frac{\abs{\tanh\left(\frac{1}{2}\tanh^{-1}\left(x\right)+\frac{1}{2}\tanh^{-1}\left(y\right)\right)-\left(\frac{1}{2}x+\frac{1}{2}y\right)}}{\tanh^{-1}\left(\alpha\right)\left(x-y\right)}\leq\frac{9}{2}\cdot
 	\]
 	This bound can be shown by differentiating with respect to $y$ to
 	the find the maximum of the left hand side.
 	
 \end{proof}
 \begin{prop}
 	\label{prop:variance_of_f}Let $f:\mathcal C_{n}\to\bR$, let $\xi$ be
 	a product measure over $\mathcal C_{n}$, and let $Y\sim\xi$. Then 
 	\[
 	\e\abs{f\left(Y\right)-f\left(\e Y\right)}\leq\sqrt{n}\text{Lip}\left(f\right).
 	\]
 \end{prop}
 
 \begin{proof}
 	Let $M_{i}=\e\left[f\left(Y\right)\mid Y_{1},\ldots,Y_{i}\right]$.
 	Then the variance of $f$ can be bounded by 
 	\[
 	\var\left[f\left(Y\right)\right]=\sum_{i=1}^{n}\e\left(M_{i}-M_{i-1}\right)^{2}\leq\text{Lip}^{2}\left(f\right)\sum_{i=1}^{n}\var\left[Y_{i}\right]\leq n\text{Lip}^{2}\left(f\right).
 	\]
 	By Jensen's inequality, 
 	\[
 	\e\abs{f\left(Y\right)-f\left(\e Y\right)}=\e\sqrt{\left(f\left(Y\right)-f\left(\e Y\right)\right)^{2}}\leq\sqrt{\e\left(f\left(Y\right)-f\left(\e Y\right)\right)^{2}}=\sqrt{\var\left[f\left(Y\right)\right]}.
 	\]
 \end{proof}
 \begin{proof}[Proof of the chain rule Lemma \ref{lem:chain_rule}]
 	For $y\in\mathcal C_{n}$ in the discrete hypercube, denote by $S_{i}\left(y\right)$
 	the vector which is equal to $y$ everywhere, except for the $i$-th
 	entry, so that 
 	\[
 	\left(S_{i}\left(y\right)\right)_{j}=\begin{cases}
 	y_{j} & i\neq j\\
 	-y_{j} & i=j.
 	\end{cases}
 	\]
 	Using this notation, we have that 
 	\begin{equation}
 	\abs{\partial_{i}\left(h\circ f\right)\left(y\right)-h'\left(f\left(y\right)\right)\partial_{i}f\left(y\right)}=\abs{-y_{i}\frac{h\left(f\left(S_{i}\left(y\right)\right)\right)-h\left(f\left(y\right)\right)}{2}-h'\left(f\left(y\right)\right)\partial_{i}f\left(y\right)}.\label{eq:using_s_for_h}
 	\end{equation}
 	Using Taylor's theorem for $h$ around $f\left(y\right)$ with the
 	Lagrange remainder, there exists a $z\in\left[f\left(y\right),f\left(S_{i}\left(y\right)\right)\right]$
 	such that 
 	\[
 	h\left(f\left(S_{i}\left(y\right)\right)\right)-h\left(f\left(y\right)\right)=\left(f\left(S_{i}\left(y\right)\right)-f\left(y\right)\right)h'\left(f\left(y\right)\right)+\frac{1}{2}\left(f\left(S_{i}\left(y\right)\right)-f\left(y\right)\right)^{2}h''\left(z\right).
 	\]
 	Putting this into equation (\ref{eq:using_s_for_h}), we get
 	
 	\begin{align*}
 	& \abs{\partial_{i}h\left(f\left(y\right)\right)-h'\left(f\left(y\right)\right)\partial_{i}f\left(y\right)}=\\
 	& =\left|-\frac{1}{2}y_{i}\left(\left(f\left(S_{i}\left(y\right)\right)-f\left(y\right)\right)h'\left(f\left(y\right)\right)+\frac{1}{2}\left(f\left(S_{i}\left(y\right)\right)-f\left(y\right)\right)^{2}h''\left(z\right)\right)-\right.\\
 	& \,\,\,\,\,\,\,\,\,\,\,\,\,\left.-h'\left(f\left(y\right)\right)\partial_{i}f\left(y\right)\right|\\
 	& =\abs{\partial_{i}f\left(y\right)h'\left(f\left(y\right)\right)-\frac{y_{i}}{4}\left(f\left(S_{i}\left(y\right)\right)-f\left(y\right)\right)^{2}h''\left(z\right)-h'\left(f\left(y\right)\right)\partial_{i}f\left(y\right)}\\
 	& =\abs{\frac{y_{i}}{4}\left(f\left(S_{i}\left(y\right)\right)-f\left(y\right)\right)^{2}h''\left(z\right)}\\
 	& =\abs{\partial_{i}f\left(y\right)}^{2}\abs{h''\left(z\right)}\\
 	& \leq BL^{2}.
 	\end{align*}
 	Equations (\ref{eq:chain_rule_one_norm}) and (\ref{eq:chain_rule_two_norm})
 	then follow immediately. 
 	
 	For equation (\ref{eq:chain_rule_expectation}), let $x\in\overline{\mathcal C_{n}}$
 	and let $\xi$ be the product measure on $\mathcal C_{n}$ such that for
 	$Y\sim\xi$, $\e Y=x$. Applying equation (\ref{eq:expectation_switcharoo_gradient})
 	on $\grad f$ and $\grad\left(h\circ f\right)$, we have
 	
 	\begin{align*}
 	\onenorm{h'\left(f\left(\e Y\right)\right)\grad f\left(\e Y\right)-\e\grad\left(h\circ f\right)\left(Y\right)} & =\onenorm{h'\left(f\left(\e Y\right)\right)\e\grad f\left(Y\right)-\e\grad\left(h\circ f\right)\left(Y\right)}\\
 	& \leq\onenorm{\e\left[h'\left(f\left(\e Y\right)\right)\grad f\left(Y\right)-h'\left(f\left(Y\right)\right)\grad f\left(Y\right)\right]}+\\
 	& \,\,\,\,\,\,\,\,\,\,+\onenorm{\e\left[h'\left(f\left(Y\right)\right)\grad f\left(Y\right)-\grad\left(h\circ f\right)\left(Y\right)\right]}.
 	\end{align*}
 	By equation (\ref{eq:chain_rule_one_norm}), the second term on the
 	right hand side is bounded by $BL^{2}n$. AS for the first term, 
 	\begin{align*}
 	\onenorm{\e\left[\left(h'\left(f\left(\e Y\right)\right)-h'\left(f\left(Y\right)\right)\right)\grad f\left(Y\right)\right]} & \leq B\e\onenorm{\abs{f\left(\e Y\right)-f\left(Y\right)}\grad f\left(Y\right)}\\
 	& \leq B\e\abs{f\left(\e Y\right)-f\left(Y\right)}nL\\
 	\left(\text{by Proposition \ref{prop:variance_of_f}}\right) & \leq BL^{2}n^{3/2}.
 	\end{align*}
 	Thus $\onenorm{h'\left(f\left(\e Y\right)\right)\grad f\left(\e Y\right)-\grad\left(h\circ f\right)\left(\e Y\right)}\leq2BL^{2}n^{3/2}$.
 \end{proof}
 \begin{proof}[Proof of Lemma \ref{lem:composition_parameters}]
 	For a vector $y\in\mathcal C_{n}$ and an index $i=1,\ldots,n$, denote
 	by $y_{i}^{+}$ the vector $y_{i}^{+}=\left(y_{1},y_{2},\ldots,y_{i-1},1,y_{i+1},\ldots,y_{n}\right)$,
 	and by $y_{i}^{-}$ the vector $y_{i}^{-}=\left(y_{1},y_{2},\ldots,y_{i-1},-1,y_{i+1},\ldots,y_{n}\right)$.
 	\begin{itemize}
 		\item \textbf{The gradient complexity}: Denote 
 		\[
 		\mathcal A_{f}=\left\{ \grad f\left(y\right):y\in\mathcal C_{n}\right\} ,\,\,\,\,\mathcal A_{h}=\left\{ \grad\left(h\circ f\right)\left(y\right):y\in\mathcal C_{n}\right\} .
 		\]
 		By equation (\ref{eq:chain_rule_two_norm}), we have that for every
 		vector $v\in\bR^{n}$,
 		\[
 		\sup_{u\in\mathcal A_{h}}\left\langle u,v\right\rangle \leq\max\left(0,B_{1}\sup_{u\in\mathcal A_{f}}\left\langle u,v\right\rangle \right)+\sqrt{n}B_{2}L_{1}^{2}\norm v_{2}.
 		\]
 		Since the expected norm of a Gaussian random vector $\Gamma$ satisfies
 		$\e\norm{\Gamma}_{2}\leq\sqrt{n}$, we get that
 		\[
 		\mathcal D\left(h\circ f\right)=\e\sup_{u\in\mathcal A_{h}}\left\langle u,\Gamma\right\rangle \leq B_{1}\mathcal D\left(f\right)+B_{2}L_{1}^{2}n.
 		\]
 		\item \textbf{The Lipschitz constant}: for every $y\in\mathcal C_{n}$ and every
 		$i=1,\ldots,n$,
 		\begin{align*}
 		\abs{\partial_{i}\left(h\circ f\right)\left(y\right)} & =\abs{\frac{h\left(f\left(y_{i}^{+}\right)\right)-h\left(f\left(y_{i}^{-}\right)\right)}{2}}\\
 		& \leq B_{1}\abs{\frac{f\left(y_{i}^{+}\right)-f\left(y_{i}^{-}\right)}{2}}\leq B_{1}L_{1}.
 		\end{align*}
 		Thus $\text{Lip}\left(h\circ f\right)\leq B_{1}L_{1}$.
 		\item \textbf{The Lipschitz constant of the gradient}: Let $x\neq y\in\mathcal C_{n}$.
 		By Lemma \ref{lem:chain_rule}: 
 		\begin{align*}
 		\norm{\grad\left(h\circ f\right)\left(x\right)-\grad\left(h\circ f\right)\left(y\right)}_{1} & =\left\Vert \grad\left(h\circ f\right)\left(x\right)-h'\left(f\left(x\right)\right)\grad f\left(x\right)+h'\left(f\left(x\right)\right)\grad f\left(x\right)-\right.\\
 		& \,\,\,\,\,\,\,\,\,\,\left.-\grad\left(h\circ f\right)\left(y\right)-h'\left(f\left(y\right)\right)\grad f\left(y\right)+h'\left(f\left(y\right)\right)\grad f\left(y\right)\right\Vert _{1}\\
 		& \leq2nB_{2}L_{1}^{2}+\norm{h'\left(f\left(x\right)\right)\grad f\left(x\right)-h'\left(f\left(y\right)\right)\grad f\left(y\right)}_{1}.
 		\end{align*}
 		The last term on the right hand side can be bounded by 
 		\begin{align*}
 		\norm{h'\left(f\left(x\right)\right)\grad f\left(x\right)-h'\left(f\left(y\right)\right)\grad f\left(y\right)} & \leq\onenorm{h'\left(f\left(x\right)\right)\grad f\left(x\right)-h'\left(f\left(x\right)\right)\grad f\left(y\right)}\\
 		& \,\,\,\,\,\,\,\,\,\,+\onenorm{h'\left(f\left(x\right)\right)\grad f\left(y\right)-h'\left(f\left(y\right)\right)\grad f\left(y\right)}\\
 		& \leq B_{1}\norm{\grad f\left(x\right)-\grad f\left(y\right)}_{1}+B_{2}\abs{f\left(x\right)-f\left(y\right)}\onenorm{\grad f\left(y\right)}\\
 		& \leq B_{1}L_{2}\norm{x-y}_{1}+B_{2}L_{1}\onenorm{x-y}L_{1}n.
 		\end{align*}
 		Putting the terms together, we get
 		\[
 		\frac{\norm{\grad\left(h\circ f\right)\left(x\right)-\grad\left(h\circ f\right)\left(y\right)}_{1}}{\onenorm{x-y}}\leq B_{1}L_{2}+3B_{2}L_{1}^{2}n.
 		\]
 	\end{itemize}
 \end{proof}
 
 \begin{proof}[Proof of Proposition \ref{prop:phi_and_g_are_close}]
 	We will show that the total variation distance between $X_{n}^{g}$
 	and $X_{\varphi}$ satisfies 
 	\[
 	\text{TV}\left(\nu,\sigma\right)\leq2\cdot2^{-n};
 	\]
 	the proof of the proposition then follows immediately. Denote by $Z_{g}$
 	and $Z_{\varphi}$ the normalizing constants of $\nu$ and $\sigma$,
 	respectively. Then 
 	\[
 	Z_{g}=Z_{\varphi}+\sum_{y\,s.t\,f\left(y\right)\leq\left(t-\delta'\right)n}e^{g\left(y\right)},
 	\]
 	and by the proof of Proposition \ref{prop:small_tail}, this implies
 	that 
 	\[
 	\eps:=\abs{Z_{g}-Z_{\varphi}}\leq2^{-n}.
 	\]
 	The total variation distance is then given by
 	\begin{align*}
 	\text{TV}\left(\nu,\sigma\right) & =\frac{1}{2}\sum_{y\in\mathcal C_{n}}\abs{\frac{e^{g\left(y\right)}}{Z_{g}}-\frac{\varphi\left(y\right)}{Z_{\varphi}}}\\
 	& =\frac{1}{2}\sum_{f\left(y\right)<\left(t-\delta'\right)n}\abs{\frac{e^{g\left(y\right)}}{Z_{g}}-\frac{\varphi\left(y\right)}{Z_{\varphi}}}+\frac{1}{2}\sum_{f\left(y\right)\geq\left(t-\delta'\right)n}\abs{\frac{e^{g\left(y\right)}}{Z_{g}}-\frac{\varphi\left(y\right)}{Z_{\varphi}}}.
 	\end{align*}
 	By definition of $\varphi$ and by Proposition \ref{prop:small_tail},
 	the first term on the right hand side is bounded by
 	\begin{align*}
 	\frac{1}{2}\sum_{f\left(y\right)<\left(t-\delta'\right)n}\abs{\frac{e^{g\left(y\right)}}{Z_{g}}-\frac{\varphi\left(y\right)}{Z_{\varphi}}} & =\frac{1}{2}\sum_{f\left(y\right)<\left(t-\delta'\right)n}\frac{e^{g\left(y\right)}}{Z_{g}}\leq\frac{1}{2}\cdot2^{-n}
 	\end{align*}
 	The second term is bounded by
 	\begin{align*}
 	\frac{1}{2}\sum_{f\left(y\right)\geq\left(t-\delta'\right)n}\abs{\frac{e^{g\left(y\right)}}{Z_{g}}-\frac{\varphi\left(y\right)}{Z_{\varphi}}} & =\frac{1}{2}\sum\varphi\left(y\right)\abs{\frac{1}{Z_{\varphi}+\eps}-\frac{1}{Z_{\varphi}}}\\
 	& =\frac{1}{2}\sum\frac{\varphi\left(y\right)}{Z_{\varphi}}\abs{\frac{1}{\left(1+\frac{\eps}{Z_{\varphi}}\right)}-1}\\
 	& \leq\frac{1}{2}\sum\frac{\varphi\left(y\right)}{Z_{\varphi}}\abs{\frac{\eps}{Z_{\varphi}}+\frac{1}{2}\frac{\eps^{2}}{Z_{\varphi}^{2}}}\\
 	& \leq\frac{1}{2}\sum\frac{\varphi\left(y\right)}{Z_{\varphi}}\abs{\frac{2\eps}{Z_{\varphi}}}\leq2^{-n}.
 	\end{align*}
 \end{proof}

\end{document}